\documentclass[12pt]{amsart}
\usepackage{amsmath, amssymb, amsthm, mathrsfs, amscd}

\newcommand{\col}{\colon \negthinspace}

\newcommand{\A}{\mathfrak A}
\newcommand{\m}{{\mathfrak{m}}}

\newcommand{\Iq}{I^{[q]}}
\newcommand{\Jq}{J^{[q]}}
\newcommand{\inc}{\subseteq}
\newcommand{\al}{\alpha}

\newcommand{\la}{\lambda}

\newcommand{\ra}{\longrightarrow}

\newcommand{\HK}{Hilbert-Kunz\,\,}
\newcommand{\tors}[3][q]{\ensuremath{\tor_{#2}(#3,R/I^{[#1]})}}

\newcommand{\len}[3][q]{\ensuremath{\lambda(\tor_{#2}(#3,R/I^{[#1]}))}}
\newcommand{\ehk}{e_{HK}}
\newcommand{\unmixed}{formally unmixed\,\,}
\DeclareMathOperator{\vol}{Vol}
\DeclareMathOperator{\Max}{Max}
\DeclareMathOperator{\hg}{\operatorname{height}}
\DeclareMathOperator{\di}{\operatorname{dim}}

\DeclareMathOperator{\tor}{Tor}
\DeclareMathOperator{\Hom}{Hom}

\DeclareMathOperator{\coker}{coker}
\DeclareMathOperator{\rank}{rank}

\theoremstyle{plain}
\numberwithin{equation}{section}
\newtheorem{theorem}{Theorem}[section]
\newtheorem{prop}[theorem]{Proposition}
\newtheorem{discuss}[theorem]{Discussion}
\newtheorem{lemma}[theorem]{Lemma}
\newtheorem{cor}[theorem]{Corollary}
\newtheorem{conj}[theorem]{Conjecture}

\newtheorem{define}[theorem]{Definition}
\newtheorem{remark}[theorem]{Remark}
\newtheorem{example}[theorem]{Example}
\newtheorem{exercise}[theorem]{Exercise}

\numberwithin{equation}{section}

\begin{document}

\title{Hilbert-Kunz Multiplicity and the F-signature}

\author{Craig Huneke}
\address{Department of Mathematics,
University of Kansas,
Lawrence, KS 66045}
\email{huneke@math.ku.edu}
\urladdr{http://www.math.ku.edu/\textasciitilde huneke}

\date{\today}
\thanks{The author was partially supported by NSF grant
DMS-1063538. I thank them for 
their support.}

\bibliographystyle{alpha}

\begin{abstract}
This paper is a much expanded version of two talks given in Ann Arbor
in May of 2012 during the computational workshop on F-singularities.
We survey some of the theory and results concerning the Hilbert-Kunz
multiplicity and F-signature of positive characteristic local rings. 
\end{abstract}

\maketitle

\begin{center}
Dedicated to David Eisenbud, on the occasion of his 65th birthday.
\end{center}

\section{Introduction}
\medskip

Throughout this paper $(R,\m,k)$ will denote a Noetherian local
ring of prime characteristic $p$ with maximal ideal $\m$ and residue field $k$. 
We let $e$ be a varying non-negative integer, and let
$q = p^e$. By  $\Iq$ we
denote the ideal generated by $x^{q}$,
$x\in I$.  If $M$ is a finite $R$-module, $M/\Iq M$ has finite length.
We will use $\lambda(-)$ to denote
the length of an $R$-module. 
We assume knowledge of basic ideas in commutative algebra, including the
usual Hilbert-Samuel multiplicity, Cohen-Macaulay, regular, and Gorenstein
rings.

The basic question this paper studies is how
$\la(M/\Iq M)$ behaves as a function on $q$, and how understanding this behavior leads to better understanding
of the singularities of the ring $R$. In a seminal paper which appeared in 1969, \cite{Ku1}, Ernst Kunz introduced the study of
this function as a way to measure how close the ring $R$ is to being regular.  

The \it Frobenius \rm homomorphism is the map $F: R\ra R$ given by $F(r) = r^p$. We say that $R$ is \it F-finite \rm if
$R$ is a finitely generated module over itself via the Frobenius homomorphism. It is not difficult
to prove that if $(R, \m, k)$ is a complete local Noetherian ring of characteristic $p$, or an affine ring over
a field $k$ of characteristic $p$, then $R$ is F-finite if and only if $[k^{1/p}:k]$ is finite.
When $R$ is reduced we can identify the
Frobenius map with the inclusion of $R$ into $R^{1/p}$, the ring of  $p$th roots of elements of $R$.
If $M$ is an $R$-module, we will usually write $M^{1/q}$ to denote what is more commonly denoted $F_*^e(M)$, where
$q = p^e$, the module which is the same as $M$ as abelian groups, but whose $R$-module structure is
coming from restriction of scalars via $e$-iterates of the Frobenius map. This is an exact functor on the category
of $R$-modules. Notice that $F^e_*(R)$ can be naturally identified with $R^{1/q}$.

If  the residue field $k$ of $R$ is perfect then the lengths of the $R$-modules $R^{1/q}/IR^{1/q}$ and $R/I^{[q]}$
are the same. If $k$ is not perfect, but $R$ is F-finite,  then we can adjust by $[k^{1/q}:k]$. We define $\alpha(R): = \operatorname{log}_p([k^{1/p}:k])$,
so that we can write $[k^{1/q}:k] = q^{\al(R)}$. With this notation, $\la_R(R^{1/q}/IR^{1/q}) = \la_R(R/I^{[q]})q^{\al(R)}$. 

More broadly, the two numbers we will study, namely the Hilbert-Kunz multiplicity and the F-signature, are
characteristic $p$ invariants which give information about the singularities of $R$, and lead to many interesting
issues concerning how to use characteristic $p$ methods to study singularities. There are four basic facts about
characteristic $p$ which make things work. Those facts are first that $(r+s)^p = r^p+s^p$ for
elements in a ring of characteristic $p$ (i.e., the Frobenius is an endomorphism); second, that the map from
$R\ra R^{1/p}$ is essentially the same map as that of $R^{1/q}\ra R^{1/qp}$ when $R$ is reduced and $q = p^e$;
third that $\sum_i \frac{1}{p^i}$ converges (!); and lastly that the flatness of Frobenius characterizes regular rings.
Virtually everything we prove comes down to these interelated facts. 

Throughout this paper, whenever possible we have tried to give new (or at least not published) approaches to basic
material. This is not done for the sake of whimsy, but to provide extra methods which may be helpful. Thus, the
approach we take to proving the existence of the Hilbert-Kunz multplicity and the F-signature, while following the
general lines of the proofs of Paul Monsky \cite{Mo1} and Kevin Tucker \cite{Tu} respectively, uses a lemma of Sankar Dutta \cite{D} as a central
point, which is not present in the usual proofs. When we present the proof of the existence of a second coefficient, we
veer from the paper \cite{HMM} to present another proof, based on the growth of the length of certain Tor modules,
due to Moira McDermott and this author. In proving the
theorem relating tight closure to  the \HK multiplicity, we use a lemma of Ian Aberbach \cite{Ab1} as a crucial point
in the proof instead of presenting
the original proof in \cite{HH1}. We provide examples of \HK multiplicities throughout the paper, but often
do not give details of the calculation.

We describe the contents of this paper. In the second section we give some early results of Kunz on the relationship
between regular local rings and the \HK function. Kunz was ahead of his time in this regard, though characteristic $p$
methods in commutative algebra were being using to study various homological conjectures at around the same time.
In section three, we develop basic results and definitions
needed to give our main existence theorems. Our main technical tool we use is a lemma of Dutta \cite{D} which
gives information about the nature of prime filtrations of $R^{1/q}$. We prove that the \HK multiplicity
exists. 
Section four proves that for \unmixed rings, the \HK multiplicity is one if and only if $R$ is regular. 
Here \it formally unmixed \rm means that for all associated primes $Q$ of the completion of a local
ring $R$, $\dim \widehat{R}/Q = \dim R$.
Section five provides the relationship between tight closure and \HK multiplicity. In section six
we prove that the F-signature exists and do some examples. Section seven proves the existence of a second
coefficient in the \HK function for normal rings. The final section takes up lower bounds on the
\HK multiplicity, introducing the volume estimates due to Watanabe and Yoshida \cite{WY2}, \cite{WY4}, as well as the method of root adjunction
of Aberbach and Enescu \cite{AE3}, \cite{AE4} and recent improvements by Celikbas, Dao, Huneke, and Zhang \cite{CDHZ}. We close with some results
of Doug Hanes \cite{Ha}.

This survey does not present the considerable research dealing with the many remarkable and difficult caluculations of
\HK multiplicity. 
For example, for work on plane cubics, see Pardue's thesis,
\cite{BC} and \cite{Mo2}. For plane curves in general see \cite{Tr2}, and for general two-dimensional graded rings
either \cite{Br1} or \cite{Tr1}. For binomial hypersurfaces, see \cite{Co} or \cite{U}. 
For flag varieties see \cite{FT}. The \HK multiplicity of  Rees algebras was the theme of \cite{EtY}. Many other important examples or work
are in 
\cite{Br1}-\cite{Br3}, \cite{Co}, \cite{E}, \cite{EtY}, \cite{GM}, \cite{Mo1}-\cite{Mo7}, \cite{MS},  \cite{MT1}, \cite{MT2},
\cite{S}, \cite{Tu}, \cite{Tr1}-\cite{Tr3}, and
\cite{WY1}-\cite{WY4}. We borrow freely from these papers for some of the examples presented in this paper.
We do not cover many new developments and calculations of the F-signature, for example see
\cite{BST1}-\cite{BST2} and for toric rings see \cite{S} and more recently \cite{VK}. See \cite{EY} for further extensions of \HK multiplicity, and \cite{Vr} for
additional work. We also do not discuss the very interesting work being done on limiting value of
\HK multiplicities as $p$ goes to infinity, For example, see \cite{BLM}, \cite{GM}, and \cite{Tr3}. For an excellent survey of other numerical invariants of singularities defined
via Frobenius and their relationship to birational algebraic geometry and the theory of test ideals, see \cite{STu}. 

\medskip

\section{Early History}

\medskip

Ernst Kunz was a pioneer in this study, realizing that studying the colengths of Frobenius powers of $\m$-primary
ideals would be an interesting idea.

\begin{theorem}\label{kunzthm} (\cite[Theorem 2.1, Proposition 3.2, Theorem 3.3]{Ku1}) Let $(R, \m, k)$ be a Noetherian local ring of dimension $d$
and prime characteristic $p > 0$.
For every $e\geq 0$, and $q = p^e$, $\la(R/\m^{[q]}) \geq  q^d$. Moreover, equality holds for some $q$ if and only if $R$ is
regular, in which case equality holds for all $q$.  If $R$ is F -finite, then $R^{1/q}$ is a free module
for some $q > 1$ if and only if
R is regular.
\end{theorem}

\begin{proof} We may complete $R$, and assume the residue field is algebraically closed to prove
the first statement. We may also go
modulo a minimal prime of $R$ to assume that $R$ is a complete local domain; this change will only potentially
decrease $\la(R/\m^{[q]})$. We claim that $R^{1/q}$ has rank $q^{d}$ as
an $R$-module in this case.
Choose a coefficient field $k$ and a minimal
reduction $x_1,...,x_d$ of the maximal ideal. Let $A$ be the complete subring $k[[x_1,...,x_d]]$ which
is isomorphic with a formal power series. 
Note that $A^{1/q}\cong k[[x_1^{1/q},...,x_d^{1/q}]]$, which
is a free $A$-module of rank $q^{d}$, whose basis is given by 
by arbitrary monomials of the form $x_1^{a_1/q}\cdots x_d^{a_d/q}$ where $0\leq a_i\leq q-1$..
Since the rank of $R$ over $A$ and
the rank of $R^{1/q}$ over $A^{1/q}$ are the same, it follows that the rank of $R^{1/q}$ over $R$ is
exactly $q^{d}$. (We note that if $R$ is an F-finite complete domain but the residue field is not perfect, then essentially the same proof shows that the
rank of $R^{1/q}$ is exactly $q^{(d+\al(R))}$.) Since $R^{1/q}$ is a finite $R$-module, $\mu_R(R^{1/q})\geq q^{d}$, with equality if and only if
$R^{1/q}$ is a free $R$-module. However, $\mu_R(R^{1/q}) = \la_R(R^{1/q}/\m R^{1/q}) = \la_R(R/\m^{[q]})$,
which implies that $\la_R(R/\m^{[q]})\geq q^d$. Notice that equality occurs in this case if and only if $R^{1/q}$ is a free
$R$-module.

If $R$ is regular, then since $\m$ is generated by a regular sequence, it easily follows that $\la(R/\m^{[q]}) =  q^d$.
The second statement also easily is seen when $R$ is regular and $R$ is F-finite; one can complete and
use the Cohen Structure theorem to do the complete case, and then descend using standard facts. It is the converse
of both statements that is the most interesting part of the theorem.

Suppose that equality holds for some $q$, i.e., $\la(R/\m^{[q]}) =  q^d$. We can complete the ring and extend the residue field to be
algebraically closed without changing this equality, so without loss of generality, $R$ is $F$-finite and $\al(R) = 0$.
Note that $\la(R/\m^{[q^n]}) = q^{nd}$ for all $n\geq 1$, by a simple induction. 

We claim that $R$ is a domain; for if $Q$ is a minimal prime of $R$ of maximal dimension, then we have that
$q^{nd} = \la(R/\m^{[q^n]})\geq \la(R/\m^{[q^n]}+Q)\geq q^{nd}$. Hence we have equality throughout. But then
$\la(R/\m^{[q^n]}) = \la(R/\m^{[q^n]}+Q)$ forces $\la((\m^{[q^n]}+Q)/\m^{[q^n]}) = 0$, so that
$Q\inc \cap_n \m^{[q^n]} = 0$. From the first part of this theorem, we then obtain that for all $n\geq 1$,
$R^{1/q^n}$ is a free $R$-module.

We next claim that $R$ is Cohen-Macaulay. Let
$x_1,...,x_d$ be a system of parameters generating an ideal $J$. Then $\la(R/J^{[q^n]}) = \la(R^{1/q^n}/JR^{1/q^n}) = \la(R/J)q^{dn}$,
since  $R^{1/q^n}$ is a free $R$-module of rank $q^{dn}$. 
By a formula of Lech \cite[Theorem 11.2.10]{SH}: $\varinjlim \lambda(R/J^{[q^n]})/q^{dn} = e(J)$,
the usual multiplicity of $J$. Hence the multiplicity of $J$ is the colength of $J$. Since
$J$ is generated by a system of parameters, it follows that $R$ is Cohen-Macaulay. (See \cite[Theorem 4.6.10]{BH}).

Now choose a system of parameters as above, and fix $n$ such that $\m^{[q^n]}\inc J$, where
$J$ is the ideal generated by the parameters. Suppose that the projective dimension of $k$
is infinite. We compute $\tor_{d+1}(R/J, R/\m^{[q^n]})$ in two ways. From the fact that
$J$ is generated by a regular sequence of length $d$, this Tor module is $0$. On the other hand, 
we can take the free resolution of $k$ and tensor with $R^{1/q^n}$ and obtain an $R^{1/q^n}$ minimal
free resolution of $R^{1/q^n}/\m R^{1/q^n}$. Identifying $R^{1/q^n}$ with $R$, we see that
a free resolution of $R/\m^{[q^n]}$ is obtained by applying the Frobenius to the maps in the
free resolution of $k$, which has the effect of raising all entries in matrices in the resolution
(after fixing bases of the free modules) to the $q^n$th powers. Now tensoring with $R/J$, we
see the homology at the $(d+1)$st stage is $0$ if and only if the projective dimension of
$k$ is at most $d$, since the maps become $0$ after tensoring with $R/J$. It follows that $R$ is
regular.

\end{proof}

\begin{exercise} {\rm If $(R,\m,k)$ is F-finite, and $Q$ is a prime ideal, prove that
$\al(R_Q) = \al(R)p^{\dim(R/Q)}$. (See \cite[Proposition 2.3]{Ku2}.)}\end{exercise}

\medskip
\begin{exercise}\label{rlr-hk}{\rm Let $(R,\m,k)$ be a regular local ring of dimension $d$ and
prime characteristic $p$, and let $I$ be an $\m$-primary ideal. Prove that
$\la(R/I^{[q]}) = q^d\la(R/I)$ so that in particular, $\ehk(I) = \la(R/I)$.}
\end{exercise}
\bigskip

\section{Basics}

We begin with some estimates on the growth of the Hilbert-Kunz function, and some
examples.

\begin{lemma}\label{basicineq}  Let $(R, \m, k)$ be a Noetherian local ring of dimension $d$
and prime characteristic $p > 0$. We let $e(I)$ denote the multiplicity of
the ideal $I$.
Let $I$ be an  
$\m$-primary ideal. Then  ($q = p^e$),
$$
e(I)/d!\leq \liminf \lambda(R/I^{[q]})/q^{d}\leq \limsup \lambda(R/I^{[q]})/q^{d}\leq e(I)
$$

\end{lemma}

\begin{proof} We can make an extension of $R$ to assume that the residue field
is infinite without changing any of the relevant lengths.
Let $J$ be a minimal reduction of $I$, so that $J$ is generated by a system of
parameters.
There are containments, $J^{[q]}\inc I^{[q]}\inc I^q$ which gives
inequalities on the lengths,
$$
\lambda(R/J^{[q]})\geq \lambda(R/I^{[q]})\geq \lambda(R/I^q).
$$
For large $q$, the right hand length is given by a polynomial in $q$
of degree $d$ with leading coefficient
$e(I)/d!$.  Dividing by $q^{d}$ gives one inequality. For the other, we use
a formula of Lech \cite[Theorem 11.2.10]{SH}: $\varinjlim \lambda(R/J^{[q]})/q^d = e(J)$. Since
$J$ is a reduction of $I$, $e(J) = e(I)$. 
\end{proof}

\begin{cor}\label{HKdimone}  Let $(R, \m, k)$ be a Noetherian local ring of dimension $1$
and prime characteristic $p > 0$ 
Let $I$ be an $\m$-primary ideal. Then $e(I) = \varinjlim \lambda(R/I^{[q]})/q^{d}$
\end{cor}

\begin{proof}  Set $d = 1$ in the above formula.
\end{proof}

\begin{example}\label{1dimex}{\rm Although the one-dimensional case may seem very transparent, as
the usual multiplicity equals the \HK multiplicity, the
actual Hilbert function is by no means obvious. Here is one example from \cite{Mo1}.
Let $k$ be a field of characteristic $p$ congruent to
$2$ or $3$ modulo $5$. Set $R = k[[X,Y]]/(X^5-Y^5)$.  $R$ is a one-dimensional
local ring with maximal ideal $\m = (x,y)$, and the multiplicity of $R$ is $5$.
The difference $|\lambda(R/\m^{[q]}) - 5q|$  
is bounded by a constant. But it is not a constant in general.
If we write the constant as $d_e$ where $q = p^e$, then when $e$ is even
$d_e = -4$ while when $e$ is odd, $d_e = -6$. For one-dimensional complete
local rings Monsky shows that
the `constant' term is a periodic function. See \cite{Mo1} for details. See also \cite{Kr} for
work in the graded case.}
\end{example}

Our goal of this section is to prove that $\varinjlim \lambda(R/I^{[q]})/q^{d}$
always exists. We call it the \it Hilbert-Kunz multiplicity\rm.
The history of how Monsky came to prove its existence is interesting. One might
think that he was inspired by the paper of Kunz, but in fact he did not know
about it when he proved the existence. The situation was additionally complicated by
the fact that Kunz had erroneously thought that the limit did not actually exist, and
proposed a counterexample in his paper. This author asked Monsky how he came to think about
it, and here is what he replied:

``Craig asked me how I was led into looking into Kunz's papers on the characterization of regular local rings in
characteristic $p$ (and defining and studying the
Hilbert-Kunz multiplicity as a result). But  that's not the order in which things occurred.

At Brandeis I was on the thesis committee of Al Cuoco, who was working in Iwasawa theory. He studied the growth of the
p-part of the ideal class group  as one moves up the levels in a tower of number fields, where the Galois group is a
product of $2$ copies of the $p$-adic integers.  I extended his results to a product of s copies; this involved the study of
modules over power series rings, with the base ring being the $p$-adics or $\mathbb Z/p\mathbb Z$. In particular I considered the following--
let $M$  be a finitely generated module over the power series ring in $s$ variables over $\mathbb Z/p\mathbb Z$, and $J$
be the ideal generated by
the $p^n$ th powers of the variables. How does the length of $M/JM$ grow with $n$? I got an asymptotic formula for this growth,
put it into a more general setting and wrote things up. In analogy with the Hilbert-Samuel terminology I intended to
speak of the Hilbert-Frobenius function and the
Hilbert-Frobenius multiplicity.

But when I showed my result to David Eisenbud he told me that it was wrong, and that Kunz had given examples in which
there wasn't an asymptotic formula. So I looked
into Kunz's papers, discovering that he had considered such questions before me.
So it was only proper to
call the function the Hilbert-Kunz function. And call the associated limiting value the Hilbert-Kunz multiplicity, even
though Kunz had thought that it needn't exist!"

To prove the existence of the \HK multiplicity, we will consider modules as well as rings.
We use a somewhat different treatment than the paper of Monsky \cite{Mo1}, organizing our approach through a lemma
proved by Dutta \cite{D}, which is not only interesting in its own right, but has the additional benefit
that we can directly apply it to show the existence of the F-signature as well.
However, in the end, all the approaches use that the map from $R$ to $R^{1/p}$ is essentially the same as $R^{1/q}$ to
$R^{1/qp}$, and that the sum of the reciprocals of the powers of $p$ converges.

\begin{lemma}\cite[see proof of Proposition, page 428]{D}\label{Frob-filter} Let $(R,\m,k)$ be a local Noetherian domain of dimension
$d$ and prime characteristic $p$. Assume that $R$ is F-finite. Then there exists a constant $C$ and a fixed
finite set of nonzero primes, $\{Q_1,...,Q_n\}$ such that for every $q = p^e$, the $R$-module
$R^{1/q}$ has a prime filtration having at most $Cq^d$ copies of $R/Q_i$ for $i\geq 1$,
and $q^{d+\al(R)}$ copies of $R$.
\end{lemma}

\begin{proof} The proof we give, similar to Dutta's proof, was shown to me by  Karen Smith, and is essentially found in Appendix 2 of
\cite{Hu}, proof of Exercise 10.4.

Use induction on $d$; the $d = 0$ case is trivial.

Fix a maximal rank free submodule $F$ of $R^{1/p}$. We know that the rank of $F$ is  $p^{d+\al(R)}$.
Let $T$ be the cokernel of the inclusion $F \subset R^{1/p}$.
Fix a prime cyclic filtration of  $T$, and extend it by $F$ to  a filtration of $R^{1/p}$.
$$
0 \subset F = M_0 \subset M_1 \subset M_2 \subset \cdots \subset M_t = R^{1/p}.
$$
Because $F$ is maximal rank, 
the prime cyclic factors $M_{i+1}/M_i = R/\A_i$ all 
have dimension strictly less than the dimension
of $R$. Let $C_i$ be the constant which (by induction) works for $R/\A_i$,
let $C$ be twice the sum of all the $C_i$, and let $\Omega$ be the collection of
the (finite) sets of primes appearing in the filtrations of all the
$(R/\A_i)^{1/q}$, as well as the prime $(0)$. We claim that $\Omega$ and
$C$ satisfy the conclusion of the problem.

By induction on $q$, we prove that
$R^{1/q}$ has a prime filtration using primes from $\Omega$, with at most
$\frac{C}{2}(1+ 1/p + ... + 1/q)q^{d+\al(R)}$ copies of each one. Assume this is true for $q$.
Take $p^{e} =  q$ roots of all the modules above.
We have a prime cyclic  filtration (except at zeroth spot, where it is obvious
how to extend to one)
of $R^{1/q}$ modules
$$
0 \subset F^{1/q} = 
M_0^{1/q} \subset M_1^{1/q} \subset M_2^{1/q}
\subset \cdots\subset M_t^{1/q} = R^{1/qp}, 
$$
where each
factor has the form $(R/\A_i)^{1/q} = R^{1/q}/\A_i^{1/q}$.

To make this into a prime cyclic 
filtration of $R$ modules, we simply refine each inclusion
$M_i^{1/q} \subset M^{1/q}_{i+1}$ of $R$ modules by
a prime cyclic filtration. This amounts to filtering
$M^{1/q}_{i+1}/M_i^{1/q} = (R/\A_i)^{1/q}$
by $R/\A_i$ prime cyclic modules.
By induction on $d$, this can be done with only primes from $\Omega$, and
appearing with multiplicities at most
$\leq C_iq^{d-1+\al(R/\A_i)} =  C_iq^{d-1+\al(R)}$.
Thus the primes appearing in this prime cycle filtration of $R^{1/qp}/F^{1/q}$
all come from $\Omega$, and each one appears at most $(\sum_i C_i)q^{d-1+\al(R)}$ times. 

To refine the $R$ submodule
$F^{1/q}$ into a prime filtration we deal with each of the free summands $R^{1/q}$
separately. By induction there are only
primes from $\Omega$ appearing and the multiplicity of $R/Q_i$ in $F^{1/q}$
is no more than $(\text{rank} F)(\frac{C}{2})(1 + 1/p + ... +1/q)(q^{d+\al(R)})$.
The total number is then at most $(\frac{C}{2}(1 + 1/p + ... +1/q)((qp)^{d+\al(R)}) +
\frac{C}{2}q^{d-1+\al(R)}\leq \frac{C}{2}(1+...+ 1/(qp))(qp)^{d+\al(R)}\leq C(qp)^{d+\al(R)}.$ \end{proof}

\smallskip

\begin{lemma}\label{growth}  Let $(R, \m, k)$ be a Noetherian local ring of dimension $d$
and prime characteristic $p > 0$.
Let $M$ be a finitely generated $R$-module. There exists a constant $C > 0$ such that for all $e\geq 0$ and any
$\m$-primary ideal $I$ of R with $\m^{[q]}\inc I$, where $q = p^e$, we have that
              $$\la(R/I\otimes_RM)\leq Cq^{\dim M}.$$
              \end{lemma}

\begin{proof} Set $t = \mu(\m)$. Since $\m^{tq} \inc  \m^{[q]}$, we see that $R/\m^{tq}\otimes_R M$ surjects onto $R/I\otimes_R M$.
Therefore $\la(R/I\otimes_R M) \leq \la(R/(\m^{tq}) \otimes_R M)$.
The Hilbert polynomial of $M$ with respect to $\m^t$ has degree $\dim(M)$. If the leading coefficient
of this polynomial is $c$, it is clear that any $C >> c$ satisfies the desired bound.  \end{proof}

\begin{lemma}\label{torsiontor}
Let $(R,\m,k)$ be a  local ring of dimension $d$ and
prime  characteristic $p$.
If $T$ is a finitely generated torsion $R$-module 
then there exists a constant $D$ such that  for all $q = p^e$, and for all $I$ containing $\m^{[q]}$, 
$\lambda(\tor^{R}_1(R/I,T)) \leq Dq^{d-1}$.
\end{lemma}

\begin{proof}  Choose a nonzerodivisor $c\in R$ which annihilates $T$, and consider an $R/(c) = A$ presentation
of $T$:
$$...\ra A^s\ra A^r\ra T\ra 0.$$ 
Let $N$ be the kernel of the surjection of $A^r$ onto $T$. Tensoring with $R/I$, we obtain an exact
sequence,
$$\tor_1^R(A^r,R/I)\ra \tor_1^R(T,R/I)\ra N/IN\ra (A/I)^r\ra T/IT\ra 0.$$
Since $N$ is torsion, Lemma \ref{growth} implies that the length of
$N/IN$ is bounded above by $Eq^{d-1}$, for some fixed constant $E$ depending
only on $N$. Thus it suffices to bound the length of $\tor_1^R(A^r,R/I)$. Notice that $r$
does not depend upon $q$ or $I$. Hence it suffices to bound the length of $\tor_1^R(A,R/I)$.
From the exact sequence $0\ra R\overset{c}\ra R\ra A\ra 0$, we obtain after tensoring with
$R/I$ that $\tor_1^R(A,R/I)\cong (I:c)/I$. 
However, the length of $(I:c)/I$ is the same as the length of $R/(I,c)$, and by Lemma \ref{growth}, this
length is bounded by $Gq^{d-1}$ for some constant $G$ depending only on $A$. \end{proof}

\begin{exercise}\label{Tor1}{\rm Prove Lemma \ref{torsiontor} with the modification that $\lambda(\tor^{R}_1(R/I,T))
\leq Dq^{\dim(T)}$
(this is not so easy).}
\end{exercise}

These lemmas have the following crucial consequence, which is a key point in the paper of Tucker \cite[Corollary 3.5]{Tu}:

\begin{cor}\label{uniformity}  Let $(R, \m, k)$ be a Noetherian local domain of dimension $d$
and prime characteristic $p$. Assume that $R$ is F-finite. There exists a
constant $C$ such that for all $q = p^e$ and all $q'= p^{e'}$ and for
all ideals $I$ containing $\m^{[q]}$,
$$ |\la(R/I^{[q']}) - (q')^{d+\al(R)}\la(R/I)|\leq C(q')^{d+\al(R)}q^{d-1}.$$
\end{cor}

\begin{proof} Fix the constant $C$ and the primes $\{Q_1,...,Q_n\}$ as in the
statement of Lemma \ref{Frob-filter}.
Then for all $q'$ there is an exact sequence,
$$0\ra R^{(q')^{d+\al(R)}}\ra R^{1/q'}\ra T\ra 0,$$
where $T$ has a prime filtration by at most $C(q')^{d+\al(R)}$ copies of each $R/Q_i$.
Tensoring with $R/I$, we see that the difference of lengths,
$|\la(R/I^{[q']}) - (q')^{d+\al(R)}\la(R/I)|$ is bounded by the sum of $\la(T/IT)$ +
$\la(\tor_1^R(T,R/I))$. This sum in turn is bounded by 
$$\sum_{i=1}^{n}C(q')^{d+\al(R)}(\la(R/(Q_i,I)) + \la(\tor_1^R(R/Q_i,R/I)).$$
To prove the Corollary it suffices to prove that there is a constant $D$, not depending on
$q$, $q'$, or $I$ such that $\la(R/(Q_i,I))\leq Dq^{d-1}$ for each $i$, and
$\la(\tor_1^R(R/Q_i,R/I))\leq Dq^{d-1}$.  The existence of such a constant $D$ follows from Lemmas \ref{growth} and \ref{torsiontor}
respectively. \end{proof}

\begin{remark}\label{exists}{\rm We can now easily prove that the Hilbert-Kunz multiplicity exists for the ring itself and arbitrary
$\m$-primary ideals $I$ in the case $R$ is an F-finite domain.
To do the general case, however, requires a little more work which one needs in any case to deal with
additivity properties of the Hilbert-Kunz multiplicity. However, it is worth seeing this easy case deduced 
from the corollary. We may assume that $k$ is algebraically closed. Set $c_q = \la(R/\Iq)/q^d$. Apply Corollary \ref{uniformity} with $I$ replaced by $\Iq$. Divide by $(q'q)^d$. We obtain that  for all $q,q'$,
$$|c_{qq'}-c_q|\leq \frac{C}{q}.$$ This inequality forces the set of $c_q$ to be a Cauchy sequence, and hence they converge.}
\end{remark}

\begin{lemma}\label{minagree}  Let $(R, \m, k)$ be a Noetherian local reduced  ring of dimension $d$
and prime characteristic $p > 0$.
Let $P_1,\ldots, P_m$ be those minimal primes of $R$ with $\dim(R/P_i) = d$. If $M$ and $N$
are finitely
generated $R$-modules such that $M_{P_i} \cong N_{P_i}$ for each $i$, 
then there exists a positive constant
$C$ such that for all $e\geq 0$  and for every  ideal $I$ of $R$ with $\m^{[q]} \inc I$, where $q = p^e$, we have
$|\la(R/I \otimes_R M) - \la(R/I \otimes_R N)| \leq  Cq^{d-1}$.
\end{lemma}

\begin{proof} Let $W = R\setminus (\cup_i P_i)$, so that  $R_W \cong  R_{P_1} \times \cdots \times R_{P_m}$,  and we have that
$M_W \cong  N_W$.
Since $(\Hom_R(M,N))_W \cong  \Hom_{R_W}(M_W, N_W)$, there is some $\phi \in  \Hom_R(M, N)$
such that $\phi_W$ is an isomorphism. Since $\coker(\phi)$ satisfies $\coker(\phi)_W = 0$ and thus has
dimension strictly smaller than $d$, we can find a positive constant $C$ such that for all $e\geq 0$ 
and for any ideal $I$ of $R$ which contains $\m^{[q]}$, we have that  $|\la(R/I \otimes_RR/\coker(\phi))| \leq  Cq^{d-1}$.
\end{proof} 

We use some well-known notation in the next few results.
Let $f,g: {\mathbb N}\rightarrow {\mathbb
R}$ be functions from the nonnegative integers to the real numbers.
Recall that  $f(n)= O(g(n))$ if there exists a positive
constant $C$ such that $|f(n)|\leq Cg(n)$ for all $n\gg 0$, and
we write $f(n)=o(g(n))$ if $\lim_{n \rightarrow \infty} f(n)/g(n)=0$.

\begin{prop}\label{HK-ses}  Let $(R, \m, k)$ be a Noetherian local ring of dimension $d$
and prime characteristic $p > 0$. 
$0\rightarrow N\rightarrow M\rightarrow K\rightarrow 0$ be a short exact
sequence of finitely generated $R$-modules. Then,
$$
\lambda(M/I^{[q]}M) = \lambda(N/I^{[q]}N) + \lambda(K/I^{[q]}K) + O(q^{d-1}).
$$
\end{prop}

\begin{proof} First suppose that $R$ is reduced. Then $M$ and $N\oplus K$
have isomorphic localizations at each minimal prime of $R$, and the claim
follows from Lemma \ref{minagree}.

If $R$ is not reduced, choose $q'$ such that (nilrad$(R))^{[q']}= 0$, and consider
the same exact sequence as a sequence of $R^{q'}$-modules. This ring is
reduced and applying the reduced case with the ideal $I^{[q']}\cap R^{q'}$ yields that
$$
\lambda(M/I^{[qq']}M) = \lambda(N/I^{[qq']}N) + \lambda(K/I^{[qq']}K) + O(q^{d-1}).
$$
Since $O(q^{d-1}) = O((qq')^{d-1})$, the Proposition is proved. 
\end{proof}

We are now able to prove the existence of the Hilbert-Kunz multiplicity:

\begin{theorem} \label{HKexists}  Let $(R, \m, k)$ be a Noetherian local ring of dimension $d$
and prime characteristic $p > 0$. 
 Let $M$ be a finitely generated
$R$-module, and let $I$ be an $\m$-primary ideal.  There is a real
constant $\al = e_{HK}(I,M)\geq 1$ such that $\lambda(M/I^{[q]}M) = \al q^d + O(q^{d-1})$.
If
$$
0\rightarrow N\rightarrow M\rightarrow K\rightarrow 0
$$
is a short exact
sequence of finitely generated $R$-modules, then
$$
e_{HK}(I,M) = e_{HK}(I,K) + e_{HK}(I,N).
$$
\end{theorem}

\begin{proof} By making a faithfully flat extension there is no loss of generality
in assuming that $R$ is a complete local ring with algebraically closed
residue field. By taking a prime filtration of $M$ and using Proposition \ref{HK-ses} 
it suffices to do the case in which $M = R/P$ for some prime $P$ of $R$. Thus there
is no loss of generality in assuming that $R$ is an F-finite domain and $M = R$ in
proving the first assertion. The second assertion follows immediately from the
first assertion and Proposition \ref{HK-ses}.

To prove the existence, we are now in the case of Remark \ref{exists}, which finishes the
proof. \end{proof}

We often supress the $R$ in $e_{HK}(I,R)$ and just write $e_{HK}(I)$. When $I = \m$, we
set $e_{HK}(M) = e_{HK}(\m, M)$, and refer to this value as the \HK multiplicity of $M$.

\medskip

\begin{example}\label{HK-biz}{\rm Unlike the usual multiplicity, the Hilbert-Kunz multiplicity is typically
not an integer. The Hilbert-Kunz function can appear quite bizarre, at least to begin with.
For example, let $R = {\mathbb Z}/5{\mathbb
Z}[x_1,x_2,x_3,x_4]/(x_1^4+\cdots+x_4^4)$,
then with $I = (x_1,...,x_4)$, $\la(R/I^{[5^e]}) = \frac{168}{61}(5^{3e}) - \frac
{107}{61}(3^e)$ by
\cite{HaMo}. Note that $R$ is a 3-dimensional Gorenstein ring with isolated
singularity.}
\end{example}

Just as in the theory of usual multiplicity, it is now easy to prove some
basic remarks on the behavior of the \HK multiplicity. In particular,
the following additivity theorem is highly useful.

\begin{theorem}\label{associativity}
Let $(R,\m,k)$ be a local Noetherian ring of dimension $d$ and prime
characteristic $p$.
let $I$ be an $\m$-primary ideal,
and let $M$ be a finitely generated $R$-module.
Let $\Lambda$ be the set
of minimal prime ideals $P$ of $R$ such that dim$(R/P) = \dim(R)$.
Then
$$e_{HK}(I,M) = \sum_{P\in \Lambda}e_{HK}(I,R/P)\lambda(M_P).$$
\end{theorem}

\begin{proof}
By Theorem \ref{HK-ses}, \HK
multiplicity is additive on short exact sequences.
Fix a prime filtration of $M$,
say
$$0= M_0\subseteq M_1\subseteq M_2\subseteq \cdots\subseteq M_n = M$$
where $M_{i+1}/M_i\cong R/P_i$ ($P_i$ a prime) for all $0\leq i\leq n-1$.
As $e_{HK}(I,R/Q) = 0$ if dim$(R/Q) < \dim(R)$,
the additivity of multiplicity
applied to this filtration shows that $e_{HK}(I,M)$ is a sum of the
$e_{HK}(I,R/P)$ for $P\in \Lambda$,
counted as many times as $R/P$ appears as
some $M_{i+1}/M_i$.
We can count this by localizing at $P$.
In this case,
we have a filtration of $M_P$,
where all terms collapse except for those
in which $(M_{i+1}/M_i)_P\cong (R/P)_P$,
and the number of such copies is
exactly the length of $M_P$.  \end{proof}

\begin{cor}\label{corrank}
Let $(R,\m,k)$ be a local Noetherian domain of dimension $d$ and prime
characteristic $p$.
Let $I$ be an $\m$-primary ideal of $R$,
and $M$ a finitely generated $R$-module.
Then $e_{HK}(I,M) = e_{HK}(I,R)\,\rank_RM$.
\end{cor}

\begin{proof} Recall that the rank of $M$ is by definition
the dimension of $M\otimes_RK$ over $K$,
where $K$ is the field of fractions of $R$.
We apply Lemma~\ref{minagree} with $W = R\setminus 0$:
if we set $r = \rank_RM$,
then $W^{-1}M\cong K^r\cong W^{-1}R^r$,
and the corollary follows.  \end{proof}

\begin{theorem}
\label{thmmultfiniteext}
Let $(R,\m,k)$ be a $d$-dimensional local Noetherian domain
of prime characteristic $p$, 
with field of fractions $K$, and
let $I$ be an $\m$-primary ideal.
Let $S$ be a module-finite extension domain of $R$
with field of fractions $L$.
Then
$$
e_{HK}(I,R) = \sum_{Q\in \Max(S), \dim S_Q = d}\frac{e_{HK}(IS_Q,S_Q)[S/Q:k]}{[L:K]}.
$$
\end{theorem}

\begin{proof}
Since $W^{-1}S\cong W^{-1}R^{[L:K]}$,
we can apply Lemma~\ref{minagree} to conclude that $e_{HK}(I,S) = e_{HK}(I,R)[L:K]$.
On the other hand,
$$
e_{HK}(I,S) = \lim_{q \to \infty} {\lambda_R(S/I^{[q]}S)/q^d}.
$$
                                                           As every maximal ideal $Q$ of $S$ contains $\m S$,
the Chinese Remainder Theorem implies that
$S/I^{[q]}S\cong \prod_{Q\in \Max(S)} S_Q/I^{[q]}S_Q$.
In particular,
$\lambda_R(S/I^{[q]}S) = \sum_{Q\in \Max(S)} \lambda_R(S_Q/I^{[q]}S_Q)
= \sum_{Q\in \Max(S)} \lambda_{S_Q}(S_Q/I^{[q]}S_Q)[S/Q:k]$.
Therefore $e_{hk}(I,S)$ equals
$$
\lim_{q \to \infty}
\sum_{Q\in \Max(S)} {\lambda_{S_Q}(S_Q/I^{[q]}S_Q)[S/Q:k]/q^d}
=
\lim_{q \to \infty} \sum_{\dim S_Q = d}
{\lambda_{S_Q}(S_Q/I^{[q]}S_Q)[S/Q:k]/q^d}.
$$
Hence
$$
e_{HK}(I,R) = \sum_{Q\in \Max(S), \dim S_Q = d}\frac{e_{HK}(IS_Q,S_Q)[S/Q:k]}{[L:K]}.
$$
\end{proof}

\begin{example}\label{veronese}{\rm Consider the Veronese subring $R$ defined by 
\[ R = k[[X_1^{i_1}\cdots X_d^{i_d} \,|\,i_1,\ldots,i_d \ge 0, \sum i_j = r]].  \]
Applying Theorem \ref{thmmultfiniteext} to $R \hookrightarrow S = k[[x,y]]$, we get 
\begin{equation} e_{HK}(R) = \frac{1}{~r~} \binom{d+r-1}{r}.  \end{equation}
In particular, if $d =2$, $r =e(A)$, then $e_{HK}(R) = \frac{e(R)+1}{2}$.}
\end{example}

For other examples, consider the quotient singularities.

\begin{example}\label{quot-sing}{\rm  See \cite[Theorem 5.4]{WY1}. Let $S$ be a regular local ring
and suppose that $G$ is a finite group of automorphisms of $S$ with invariant ring $R$ with
maximal ideal $\m$. By Theorem \ref{thmmultfiniteext} and Exercise \ref{rlr-hk}, one sees
that $e_{HK}(R) = \frac{1}{|G|}\la(S/\m S)$. 

This formula is used, together with a lot more work, by Watanabe and Yoshida to  give the following formulas for
the \HK multiplicities of the famous double points below:  Let $(R,\m) = k[[x,y,z]]/(f)$ 
where $f$ is one of the following:
\medskip

\begin{tabular}{lllll}
\underline{type} & \underline{equation} & \underline{char $R$} & $e_{HK}(R)$ \\
($A_n$) & $f = xy+z^{n+1}$ & $p\geq 2$ & $2-1/(n+1)$ & ($n \geq 1$)\\
($D_n$) & $f = x^2+yz^2+y^{n-1}$ & $p \geq 3$ & $2-1/4(n-2)$ & ($n\geq4$) \\
($E_6$) & $f = x^2+y^3+z^4$ & $p\geq 5$ & $2-1/24$ \\
($E_7$) & $f = x^2+y^3+yz^3$ & $p \geq 5$ & $2-1/48$ \\
($E_8$) & $f = x^2+y^3+z^5$ & $p\geq 7$ & $2-1/120$
\end{tabular}
\smallskip

Each of these hypersurfaces is the invariant subring by a finite subgroup $G \inc SL(2, k)$ which acts on the
polynomial ring $k[x, y]$. We have that  $e_{HK}(R) = 2 - 1/|G|$; see \cite[Theorem
5.1]{WY1}.}
\end{example}

\begin{example}\label{cubics}{\rm Let $S = k[x,y,z]$ where $k$ is a field of characteristic at
least five. Let $h\in S$ be homogeneous of degree $3$. Set $R = S/(h)$, and let $\m = (x,y,z)R$.
If $h$ is smooth, then $e_{hk}(\m) = \frac{9}{4}$, while if $h$ is a nodal or cuspidal cubic, 
$e_{HK}(\m) = \frac{7}{3}$. This has been done various ways. Pardue in his thesis did the
nodal cubic; see also Buchweitz and Chen \cite{BC}, Brenner \cite{Br3}, Monsky \cite{Mo8}, 
and Trivedi \cite{Tr1}, and in characteristic $2$, \cite{Mo2}.}
 \end{example}

Here are a few more examples, which we leave as an exercise:

\begin{exercise}{\rm We consider quadric hypersurfaces in $\mathbb P^3$.
Let $k$ be a field of characteristic $p > 2$, and let
Let $R$ be one of the following rings$:$
\begin{equation}
\begin{cases}
k[[X,Y,Z,W]]/(X^2),    & \mbox{if}\;\; \rank(q) =1, \\
k[[X,Y,Z,W]]/(X^2-YZ), & \mbox{if}\;\; \rank(q) =2, \\
k[[X,Y,Z,W]]/(XY-ZW),  & \mbox{if}\;\; \rank(q) =3. \\
\end{cases}
\end{equation}
Prove that $\ehk(R) = 2$, $\frac{3}{~2~}$,
or $\frac{4}{~3~}$, respectively.}
\end{exercise}

For a long time it was thought that the \HK multiplicity would always be a rational number.
All the known examples were rational, e.g., for rings of finite Cohen-Macaulay type
(see \cite{Se}) or more generally $F$-finite type (\cite{SVB}, and \cite{Y2}), for many computed hypersurfaces, for binomial hypersurfaces
(\cite{Co}), for graded normal rings of dimension two (\cite{Br2},\cite{Tr1}), and others.
However, in recent years Monsky has given convincing evidence that this will not true, though
as of the writing of this paper, there is only overwhelming evidence, but not a proof. One example
given by Monsky is the following:

\begin{example}\label{irrational}{\rm 
Let $F$ be a finite field of characteristic 2 and $h =  x^3+y^3+xyz\in F[[x,y,z]]$. Then Monsky conjectures, with a huge amount of
evidence, that the \HK multiplicity of the hypersurface $uv+h = 0$ is $\frac{4}{3}+\frac{5}{14\sqrt{7}}$,
Even more, it appears that transcendental \HK multiplicities exist. We refer to \cite{Mo6} and \cite{Mo7} for details.}
\end{example}

\bigskip
\section{Hilbert-Kunz Multiplicity Equal to One}
\medskip
We begin this section with an easy, but crucial estimate on the size of
Hilbert-Kunz functions which was observed by Hanes \cite{Ha}.

\begin{lemma}\label{filter} Let $(R,\m,k)$ be a Noetherian local ring of
dimension $d$ and prime characteristic $p0$. Let $I \subseteq J$ be two ideals with $I$
$\m$-primary (we allow $J = R$). Then
$\lambda (R/I^{[q]}) \leq  \lambda (J/I) \cdot \lambda (R/\m^{[q]})+
\lambda(R/J^{[q]})$.  
\end{lemma}

\begin{proof} Set $s = \lambda (J/I)$. Take a filtration of $I
\subseteq J \subseteq R$
$$
I=J_0 \subsetneq J_1 \subsetneq J_2 \subsetneq \cdots  \subsetneq J_s
= J  \subseteq R
$$
so that $\lambda (J_i/J_{i-1}) =1$ i.e. $J_i/J_{i-1} \cong R/\m, \
\forall i=1,2,\dots,s.$ That is to say $J_i = (J_{i-1}, x_i)$ for some
$x_i \in J_i$ such that $J_{i-1}:x_i = \m.$

For every $q=p^e,$ there is a corresponding filtration of $I^{[q]}
\subseteq J^{[q]} \subseteq R$
$$
I^{[q]}=J_0^{[q]} \subseteq J_1^{[q]} \subseteq J_2^{[q]} \subseteq
\cdots  \subseteq J_s^{[q]} = J^{[q]}  \subseteq R,
$$
where $J_i^{[q]}/J_{i-1}^{[q]} \cong R/(J_{i-1}^{[q]}: x_i^q)$, which is
a homomorphic image of $R/m^{[q]},$ for every $i=1,2,\dots, s.$ So
$\lambda (J_i^{[q]}/J_{i-1}^{[q]}) \leq \lambda (R/\m^{[q]}).$ Therefore
$\lambda (R/I^{[q]}) \leq  \lambda (J/I) \cdot \lambda (R/\m^{[q]})+
\lambda(R/J^{[q]}).$ \end{proof}

\begin{cor}\label{est}  Let $(R,\m,k)$ be a Noetherian local ring of dimension $d$ and prime
characteristic $p$. Let $I$ be a $\m$-primary ideal of $R$. Then
$\lambda (R/I^{[q]}) \leq  \lambda (R/I) \cdot \lambda (R/\m^{[q]}).$
If $I\inc J$ then $e_{HK}(I,R)\leq \la(J/I)e_{HK}(R) + e_{HK}(J,R)$.
\end{cor}

\begin{proof} To prove the first statement, we take $J = R$ and apply Lemma \ref{filter}.
For the second statement, the Corollary  follows from Lemma \ref{filter} by
dividing by $q^{d}$ and then taking the limits. \end{proof}

Our goal is to prove that regularity is characterized by the \HK multiplicity being one,
if the ring is \unmixed. This condition is necessary by the easy exercise below. Our
treatment is taken directly from \cite{HY}.

\begin{exercise} Let $R = k[[x,y,z]]/(xz,xy)$, where $k$ is a field of characteristic $p$.
Prove that $e_{HK}(R) = 1$.
\end{exercise}

\begin{theorem}\label{thminequality}  Let $(R,\m)$ be a Noetherian local ring of dimension $d$ and prime
characteristic $p$. Let $J$ be an ideal such that $\dim R/J =1$ and $\hg J = d-1$.
Assume that $x\in R$ is a  non-zerodivisor
in $R/J$, and set $I = (J,x)$. Assume that $R_P$ is regular for every minimal
prime $P$ above $J$. Then
$ e_{HK}\left(I, R\right) \geq \lambda (R/I).$
\end{theorem}

\begin{proof} Using the properties of the usual multiplicity of
parameter ideals, the associativity formula for the usual multiplicity,
and we have
$$e_{HK}(I,R) = \lim_{q \to \infty}\frac 1{q^d} \cdot \lambda
(R/I^{[q]}) =
\lim_{q \to \infty}\frac 1{q^d} \cdot \lambda (R/(J^{[q]},x^q))$$
$$\geq \lim_{q \to \infty}\frac 1{q^d} \cdot e(x^q; R/J^{[q]}) =
\lim_{q \to \infty}\frac q{q^d} \cdot e(x; R/J^{[q]}) =
\lim_{q \to \infty}\frac 1{q^{d-1}} \cdot e(x; R/J^{[q]})$$
$$=\lim_{q \to \infty}\frac 1{q^{d-1}} \cdot \sum_{P \in \min(R/J)}
e(x;R/P) \cdot \lambda_{R_P}(R_P/J_P^{[q]})$$
$$=\lim_{q \to \infty}\frac 1{q^{d-1}} \cdot \sum_{P \in \min(R/J)}
e(x;R/P) \cdot q^{d-1} \cdot \lambda_{R_P}(R_P/J_P) $$
$$=\lim_{q \to \infty} \sum_{P \in \min(R/J)}
e(x;R/P) \cdot \lambda_{R_P}(R_P/J_P)$$
$$=\sum_{P \in \min(R/J)}e(x;R/P) \cdot \lambda_{R_P}(R_P/J_P)=
e(x;R/J)=\lambda (R/(J,x))
= \lambda (R/I).$$ 
\end{proof}

Observe that after we prove that $e_{HK}(R) = 1$ implies
the regularity of $R$, then regularity forces $e_{HK}(I) = \lambda(R/I)$ for
all $\m$-primary ideals $I$, by using the work above.

\medskip
A critical step in proving the main result of this section is in constructing an
$\m$-primary ideal $I\inc \m^{[p]}$ such that $e_{HK}(I)\geq \la(R/I)$.
This was proved by Watanabe and Yoshida \cite[Theorem 1.5]{WY1}
but in a different way than is done here.

\begin{theorem}\label{HK=1} Let $(R,\m,k)$ be a formally unmixed
Noetherian
local ring of dimension $d$ and prime characteristic $p$.
Then $e_{HK}(R)=1$ if and only if  $R$ is regular.
\end{theorem}

\begin{proof} We have already observed that if $R$ is regular, then the Hilbert-Kunz multiplicity is one.
We prove the converse.
Since the Hilbert-Kunz multiplicity
of $R$ is the same as that of its completion, we may assume $R$ is
complete.  The additivity formula for Hilbert-Kunz multiplicity Theorem \ref{associativity} shows
that
$e_{HK}(R) = \sum _P e_{HK}(R/P)\cdot \la(R_P)$
where the sum is over all minimal primes of maximal dimension.
Since $e_{HK}(R) = 1$, we deduce that $R$ can have only one minimal
prime $P$ and $R_P$ has to be field, i.e. $P_P=0.$ Hence $P=0$ since
$R\setminus P$ consists of non-zero divisors. Thus $R$ is a domain.

It suffices to prove that $\la(R/\m^{[p]})\leq p^d$ (where $d = \di(R)$) as then Theorem \ref{kunzthm}
gives that $R$ must be regular.

The singular locus of $R$ is closed and not equal to Spec$(R)$. It follows
that we can choose a prime $P$ such that dim$(R/P) = 1$ and $R_P$ is regular,
which we leave as an exercise for the reader.
Since the intersection of the symbolic powers of $P$ is zero and $R$ is
complete, Chevalley's lemma gives that some sufficiently large
symbolic power of $P$ lies inside $\m^{[p]}$. Call this symbolic power $J$.
Choose $x\in \m^{[p]}$ such that $x\notin P$. The ideal $I = (J,x)$ lies
in $\m^{[p]}$ and satisfies the hypothesis of Theorem  \ref{thminequality}. Hence
$$e_{HK}(I) \geq \la(R/I).$$

On the other hand we have $e_{HK}(I,R) \leq \lambda(\m^{[p]}/I) \cdot
e_{HK}(R) + e_{HK}(\m^{[p]},R) = \lambda(\m^{[p]}/I) +
e_{HK}(\m^{[p]},R) \leq \lambda(\m^{[p]}/I) + \lambda(R/\m^{[p]}),$ by
Lemma \ref{filter} and Corollary \ref{est}.

That is to say
 \begin{align} 
\lambda(\m^{[p]}/I) + \lambda(R/\m^{[p]}) =&\lambda (R/I) \leq e_{HK}(I,R) \\
\leq &\lambda(\m^{[p]}/I) + e_{HK}(\m^{[p]},R) \\
\leq &\lambda(\m^{[p]}/I) + \lambda(R/\m^{[p]}),
\end{align} 
which forces $\lambda(R/\m^{[p]}) = e_{HK}(\m^{[p]},R)$. However,
$$e_{HK}(\m^{[p]},R) = \varinjlim \frac {\la(R/\m^{[pq]})} {q^d} =
\varinjlim \frac {p^d\cdot\la(R/\m^{[pq]})} {(pq)^d} = p^d\cdot e_{HK}(R) = p^d.$$
Together the equalities  imply that $\lambda(R/\m^{[p]}) = p^d$, which
implies that $R$ is regular by Theorem \ref{kunzthm}. \end{proof}

The basic filtration lemmas, together with Kunz's theorem already give a better result, \it provided \rm
the ring is Cohen-Macaulay. In fact, this is one of the more subtle and difficult points, to prove
that \HK multiplicity near one should imply that the ring is Cohen-Macaulay. A crucial step is provided
by results of Goto and Nakamura, see \cite{GN}. They prove the following  
beautiful generalization of the result of Serre which proves that
the multiplicity of a parameter ideal is its colength if and only if the ring is Cohen-Macaulay.

\begin{theorem}\cite{GN}\label{goto-nakamura} Let $(R,\m,k)$ be an unmixed  Noetherian local ring of prime
characteristic $p$ which is the homomorphic image of a Cohen-Macaulay local ring. Let $J$ be an
ideal generated by a system of parameters. If $\la(R/J^*) = e(J)$, then $R$ is F-rational (and
therefore is Cohen-Macaulay).
\end{theorem}

The general philosophy is that the closer the \HK multiplicity
is to one, the better the singularities of the ring. The following proposition was proved by
Blickle and Enescu, using results of Goto and Nakamura and Watanabe and Yoshida to first obtain that
the ring is Cohen-Macaulay. We state the full result here, but only give the proof assuming Cohen-Macaulay. 

\begin{prop}\label{BE-est}\cite{BE} Let $(R,\m)$ be a Noetherian local ring of dimension $d$ and prime
characteristic $p$. If $R$ is not regular, then $e_{HK}(R) >  1 + \frac{1}{p^dd!}$.
\end{prop}

\begin{proof} We give the proof assuming that $R$ is Cohen-Macaulay. Let $e$ be the multiplicity of $R$.
We may assume the residue field is infinite. Fix a minimal reduction $K$ of the
maximal ideal. We apply Corollary \ref{est} with $I = K^{[p]}$ and $J = \m^{[p]}$. This gives
that $ ep^d = e_{HK}(K^{[p]})\leq \la(\m^{[p]}/K^{[p]})e_{HK}(R) + e_{HK}(\m^{[p]}) = 
\la(\m^{[p]}/K^{[p]})e_{HK}(R) + p^de_{HK}(R).$
By Theorem \ref{kunzthm}, $\la(\m^{[p]}/K^{[p]}) = ep^d - \la(R/\m^{[p]}) \leq ep^d - (p^d+1)$
because $R$ is not regular. Putting these inequalities together and cancelling terms yields
that $ep^d\leq (ep^d-1)e_{HK}(R)$ or $1 + \frac{1}{ep^d-1}\leq e_{HK}(R).$
Since $e/d!\leq e_{HK}(R)$, if $e > d!$, then $1 + \frac{1}{d!} < e_{HK}(R)$, a stronger
statement than what we claim. Otherwise, $ep^d-1 < p^dd!$, and the proposition follows. \end{proof}

\smallskip

The reader should ask themselves where the assumption that $R$ is Cohen-Macaulay is used in the above
proof.

The methods in this section also give a proof of a result of Kunz concerning the behavior of
\HK multiplicity under specialization. It is still an open problem whether or not \HK multiplicity
is upper semi-continuous. See, however, the interesting papers of Shepherd-Barron \cite{SB} (but
be careful--Corollary 2 is not quite correct) and Enescu and Shimomoto \cite{ES}.
 \smallskip

\begin{prop}\label{semicont} \cite[Cor. 3.8]{Ku2}
Let $(R, \m,k)$ be a Noetherian local ring of dimension $d$ and prime
characteristic $p$, and let $P$ be a prime ideal of $R$ such that
$\hg(P) + \di(R/P) = \di(R)$. Then $e_{HK}(R_P)\leq e_{HK}(R).$  In
fact,  if $t = \di(R/P)$, then  $q^t \cdot ~\lambda_{R_P}((R/P^{[q]})_P)
\leq \lambda(R/\m^{[q]})$ for  every  $q=p^e.$
\end{prop}

\begin{proof} By induction, it is enough to prove the case where 
$\hg(P) =\di(R) - 1$. Notice it suffices to prove the second inequality.

Choose $f \in \m - P$. Then, using the properties of the usual
multiplicity of parameter ideals, the associativity formula for the
usual multiplicity, we have,  for all $q=p^e$,
\begin{align}
\lambda(R/(P, f)^{[q]}) &= \lambda(R/(P^{[q]}, f^q))\\
& \geq e(f^q; R/P^{[q]})\\
& =\lambda _{R_P}((R/P^{[q]})_P) \cdot e(f^q; R/P)\\
& =\lambda _{R_P}((R/P^{[q]})_P) \cdot q  \cdot \lambda(R/(f, P)).
\end{align}
By Corollary \ref{est}, we know  that
$\lambda(R/(f, P)) \cdot \lambda(R/\m^{[q]})\geq \lambda(R/(P, f)^{[q]}).$
Hence  $\lambda(R/\m^{[q]}) \geq q \cdot \lambda_{R_P}((R/P^{[q]})_P)$
for every $q=p^e.$  \end{proof}
\bigskip

\section{Hilbert-Kunz Multiplicity and Tight Closure}
\medskip

There is almost an exact parallel between the relationship of integral closure to the usual
Hilbert-Samuel multiplicity, and the relationship between tight closure to Hilbert-Kunz multiplicity.
Just as in the case of the Hilbert-Samuel multiplicity, this relationship is important both
theoretically and necessary to fully understand multiplicity. 
We use a key result of Aberbach \cite{Ab1} to make the proofs easier than the original
proof in \cite{HH1}.

Let $R^o$ denote the complement
of the union of all minimal primes of a ring $R$. The definition of tight
closure for ideals is:

\begin{define}\label{tc} Let $R$ be a Noetherian ring of prime characteristic $p$. Let $I$
be an ideal of $R$. An element $x\in R$ is said to be in the tight closure of $I$ if
there exists an element $c\in R^o$ such that for all
large $q= p^e$, $cx^q\in I^{[q]}$.
\end{define}

There is also a definition of the tight closure of submodules of finitely generated $R$-modules,
which we do not use in these notes. Of particular interest
are rings in which every ideal is tightly closed.

\begin{define} A Noetherian ring in which every ideal is tightly closed is called
\it weakly F-regular \rm.  A Noetherian ring $R$ such that $R_W$ is weakly F-regular for
every multiplicative system $W$ is called \it F-regular \rm.
\end{define}

We list a few of the main properties satisfied by tight closure.

\begin{prop}\label{tc-basic} Let $R$ be a Noetherian ring of prime characteristic
$p$ and let $I$ be an ideal.
\begin{enumerate}
\item $(I^*)^*= I^*$.  If
$I_1\subseteq I_2\subseteq R$, then $I_1^*\subseteq I_2^*$.

\item If $R$ is reduced or if $I$ has positive height, then $
x\in R$ is in $I^*$ if and only if there exists $c\in R^o$ such that $cx^q\in I^
{[q]}$ for all $q=p^e$.

\item  An element $x\in R$ is in $I^*$ iff the image of $x$ in $R/P$ is in the
tight closure of $(I+P)/P$ for every minimal prime $P$ of $R$.
\end{enumerate}

\end{prop}

 \begin{proof} Part (1) and (2) follow immediately from the definition.

 We prove (3).  One direction is clear: if $x\in I^*$, then this remains true
 modulo every minimal prime of $R$ since $c\in R^o$.
 Let $P_1,...,P_n$ be the minimal primes of $R$. If $c_i' \in R/P_i$ is
 nonzero we can always lift $c_i'$ to an element $c_i\in R^o$ by using the Prime
 Avoidance theorem. Suppose that $c_i'\in R/P_i$ is nonzero and such that $c_i'x_i^q\in I_i^{[q]}$
 for all large $q$, where $x_i$ (respectively $I_i$) respresent the images of
 $x$ (respectively $I$) in $R/P_i$.  Choose a lifting $c_i\in R^o$ of $c_i'$.
 Then  $c_ix^q\in I^{[q]} + P_i$ for every $i$.
 Choose elements $t_i$ in all the minimal primes except $P_i$. Set $c = \sum_ic_it_i$.
 It is easy to check that $c\in R^o$. Choose $q' \gg 0$ so that $N^{[q']} = 0$,
 where $N$ is the nilradical of $R$. Then $cx^q\in I^{[q]} + N$, and so
 $c^{q'}x^{qq'}\in I^{[qq']}$, which proves that $x\in I^*$.  \end{proof}

One direction of our main result of this section is quite easy from the definition:

\begin{prop} Let $(R,\m,k)$ be a Noetherian local ring of dimension $d$ and prime
characteristic $p$. Let $I$ be an $\m$-primary ideal, and suppose that $I\inc J\inc I^*$.
Then $e_{HK}(I) = e_{HK}(J)$.
\end{prop}

\begin{proof} By assumption there is an element $c\in R^o$ such that $c$ annihilates the
modules $\Jq/\Iq$ for all large $q = p^e$. These modules have a bounded number of
generators, say $t$, given by the number of generators of $J$. In particular, $(R/(c,\Jq))^t$
maps onto $\Jq/\Iq$, so that the length is at most $t\cdot\la(R/(c,\Jq))$. However, the length
of $R/(c,\Jq)$ 
is at most $O(q^{d-1})$ since the dimension of $R/(c)$ is $d-1$. It follows that
$|\la(R/\Jq)-\la(R/\Iq)| = O(q^{d-1})$, and so $e_{HK}(I) = e_{HK}(J)$. \end{proof}

The main result of this section is the following:

\begin{theorem}\label{HK-tc} Let $(R,\m,k)$ be a Noetherian local ring of dimension $d$ and prime
characteristic $p$ which is formally unmixed.
Let $I\inc J$ be $\m$-primary ideals. Then 
$e_{HK}(I) = e_{HK}(J)$ if and only if $J\inc I^*$.
\end{theorem} 

\begin{proof} One direction has already been done. To prove the other, we first observe that
for $\m$-primary ideals $K$, $e_{HK}(K) = e_{HK}(\widehat K)$ and $(\widehat{K})^* = \widehat{K^*}$.
We leave this latter equality as an exercise (see also \cite[Proposition 4.14]{HH1}).
Hence we may assume that $R$ is complete. Suppose that $e_{HK}(I) = e_{HK}(J)$. We need to prove
that $J\inc I^*$. If not, there exists a minimal prime $P$ of $R$ such that the image of $J$
in $R/P$ is not in the tight closure of the image of $I$ in $R/P$, by Proposition \ref{tc-basic}. By the additivity formula
for \HK multiplicity, Proposition \ref{associativity}, as well as our assumption that $R$
is formally unmixed, we must have that
$e_{HK}((I+P)/P) = e_{HK}((J+P)/P)$. Hence we may assume that $R$ is a complete local domain.

Suppose by way of contradiction that $J$ is not in $I^*$. We may assume that $J = (x,I)$ for
some $x\notin I^*$. We now use a result of Aberbach \cite{Ab1}: since $x\notin I^*$, 
there exists a fixed integer $k$ such that for all $q= p^e$, $\Iq:x^q\inc \m^{\lfloor q/k \rfloor}$. But now for all large enough $q$, $\la(R/\Iq) - \la(R/(\Iq, x^q)) =
\la(R/(\Iq:x^q))\geq \la(R/\m^{\rfloor q/k \lfloor})\geq \delta q^d$, where $\delta$ is any positive real
strictly less than $\frac{e}{d!k}$, where $e$ is the
multiplicity of $R$. This proves that $e_{HK}(I) \ne e_{HK}(J)$, a contradiction. \end{proof}

\medskip

With this tight closure characterization of the \HK multiplicity, we can give an important estimate on
it in the case the ring is Gorenstein, but not F-rational, meaning that systems of parameters are not
tightly closed. The is due to Blickle and Enescu \cite{BE}.  

\begin{prop}\label{nonFrat-est} Let $(R,\m,k)$ be a Noetherian local ring of dimension $d$ and prime
characteristic $p$ which is Gorenstein but not F-rational. Set $e$ equal to the multiplicity of $R$.
Then $e_{HK}(R)\geq 1 + \frac{1}{e-1}$.
\end{prop}

\begin{proof} We may assume that the residue field is infinite. Choose a minimal reduction of the
maximal ideal and let $J$ be the ideal generated by that reduction. Observe that $\la(R/J) = e$.
Since $R$ is not F-rational, $J^*\ne J$. We use Lemma \ref{Frob-filter} to see that
$$e = e_{HK}(J) =  e_{HK}(J^*) = \la(R/J^*)e_{HK}(R) \leq (e-1)e_{HK}(R)$$
giving the result. \end{proof}

If $e > d!$, then since $e_{HK}(R) \geq e/d!$, we see that $e_{HK}(R) \geq 1 + \frac{1}{d!}$.
On the other hand, if $e\leq d!$, then $e-1< d!$, and Proposition \ref{nonFrat-est} shows that
in the Gorenstein but not F-rational case, we have the same estimate that $e_{HK}(R) \geq 1 + \frac{1}{d!}$.

\begin{remark} {\rm It is worth noting that the relationship between the Hilbert-Kunz multiplicity of
ideals and the tight closure was an important idea in the construction by Brenner
and Monsky \cite{BM} of a counterexample to the localization problem in tight closure theory.}
\end{remark}
\bigskip
\section{F-signature}
\bigskip

The work of Hochster and Roberts on the Cohen-Macaulayness of rings of invariants \cite{HoR}
focused attention on the splitting properties of the map from $R$ to $R^{1/p}$. If 
$R$ is F-finite, then this map splits as a homomorphism of $R$-modules if and only if
$R$ is F-pure, i.e. the Frobenius homomorphism is a pure map. Thus the idea of splitting
copies of $R$ out of $R^{1/p}$ clearly had something to say about the singularities of $R$.
This idea was further explored during the development of tight closure, with the concept of
strong F-regularity. In \cite{SVB}, Smith and Van den Bergh studied the asymptotic behavior of summands of $R^{1/q}$ for
rings of finite F-representation type which are strongly F-regular. Yao \cite{Y1} later removed the
assumption of strong F-regularity from their work. For free summands, in \cite{HL}, the idea of the F-signature was introduced as a way
to asymptotically key track of the number such summands of $R^{1/q}$ as $q$ varies. As it turns out, almost the exact same
ideas were introduced at the same time by Watanabe and Yoshida \cite{WY5} in their study of minimal Hilbert-Kunz multiplicity.
The F-signature provides delicate information about the singularities of $R$, as we shall see. One
immediate problem was to show that a limit exists in this asymptotic construction. 
   When $R$ is Gorenstein, this was done in \cite{HL}, and we reproduce that argument here
since it is not difficult and has the additional benefit of expressing the F-signature as a difference of the \HK multiplicities of two ideals.
The case when $R$ is not Gorenstein proved to be considerably harder. After
many partial results (see \cite{Ab2}, \cite{Y2}, for example) Kevin Tucker recently proved the limit always exists. We
give a modified version of  his proof here.

We first set up the basic ideas.
Let $(R,\m,k)$ be a $d$-dimensional reduced Noetherian local ring with prime
characteristic $p$  and residue field $k$. We assume that $R$ is F-finite. 
By $a_q$ we denote the largest rank of a free
$R$-module appearing in a direct sum decomposition of $R^{1/q}$, where as usual $q = p^e$. We write
$R^{1/q} \cong  R^{a_q} \oplus M_q$ as an $R$-module, where $M_q$ has no free direct summands. 
The number $a_q$ is called the
\it e-th Frobenius splitting number \rm of $R$.

\begin{define}\label{F-sig} The F-signature of $R$, denoted $s(R)$, is
$s(R) = \varinjlim \frac{a_q}{q^{d+\alpha(R)}}$, the limit taken as $q$ goes to infinity, provided the
limit exists.
\end{define}

We first prove that the limit exists in the Gorenstein case, partly due to the ease of the proof,
and partly due to the fact that it gives a precise value for the F-signature in terms of  Hilbert-Kunz multiplicities.
This theorem is found in \cite{HL}.

\begin{theorem}\label{Fsig-Gor} Let $(R,\m,k)$ be a Noetherian local reduced Gorenstein ring of
dimension $d$ and prime characteristic $p$. Then $\varinjlim \frac{a_q}{q^{d+\alpha(R)}}$ exists
and is equal to the difference between the Hilbert-Kunz multiplicity of the ideal
$I$ generated by a system of
parameters, and the Hilbert-Kunz multiplicity of the ideal $I:m$.
\end{theorem}

\begin{proof} Let $I
= (x_1,...,x_d)$ be generated by a system of
parameters.  We claim that the difference $\lambda(M/ IM) -
\lambda(M/(I:\m)M)$ is zero for all maximal Cohen-Macaulay modules $M$ without a free summand. We state this as a separate lemma.

\begin{lemma}\label{socle} Let $(R,\m)$ be a Gorenstein local ring and
let
$M$ be a maximal Cohen-Macaulay $R$-module without a free summand. Let
$I$ be an ideal generated by a  system of parameters for $R$, and let $\Delta \in R$ be a
representative for the socle of $R/I$.  Then $\Delta M \subseteq IM$.\end{lemma} 

\begin{proof} Choose generators $\{m_1, \ldots, m_n\}$ for $M$ and
define a
homomorphism $R \to M^n$ by $1 \mapsto (m_1, \ldots, m_n)$.  Let $N$ be
the cokernel, so that we have an exact sequence
$$ 0 \to R \to M^n \to N \to 0.$$
Since $M$ has no free summands, this exact sequence
is nonsplit.  This implies, since $R$ is Gorenstein, that $N$ is not
Cohen-Macaulay. 
When
we kill $I$, therefore, there is a nonzero $\tor$:
$$0 \to \tor_1^R(N,R/I) \to \overline R \to \overline M^n \to
\overline N \to 0.$$
Since the map $\overline R \to \overline M^n$ has a nonzero kernel, we must
have $\overline\Delta \mapsto 0$.  Since the elements $m_1, \ldots, m_n$
generate $M$, this says precisely that $\Delta M \subseteq
IM$. 
\end{proof}

Returning to the proof of Theorem \ref{Fsig-Gor}, we write $R^{1/q} = R^{a_q}\oplus M_q$,
where $M_q$ is a maximal Cohen-Macaulay module without free summands. Applying Lemma \ref{socle}, we then
see that 
$q^{\alpha(R)}(\lambda(R/I^{[q]}) - \lambda(R/(I,\Delta)^{[q]})) = a_{q}$
and therefore 
$$e_{HK}(I,R) - e_{HK}((I,\Delta),R) = s(R).$$ \end{proof}

\begin{remark} {\rm The proof above shows that the F-signature of a Gorenstein
local ring is $0$ if and only if for some (or equivalently for all) ideals
$I$ generated by a system of parameters, $e_{HK}(I) = e_{HK}(I:\m)$.
As we have seen, this equality holds if and only if $I$ and $I:\m$ have the same
tight closure, which is true if and only if $I$ is not tightly closed, since every 
ideal properly containing $I$ must contain $I:\m$. Thus the F-signature is positive
in this case if and only if $R$ is F-rational (and then is strongly F-regular, as $R$ is
Gorenstein.) Aberbach and Leuschke \cite{AL} proved in general that the F-signature is
positive if and only if $R$ is strongly F-regular.
In fact the ideas of the proof above extend to prove something a little less than
strong F-regularity, 
namely, that \cite[Theorem 11]{HL} if the lim sup of 
$a_q/q^d$ is positive, then $R$ must be weakly F-regular, and in particular is
Cohen-Macaulay and integrally closed. Thus, if $R$ is not weakly F-regular, $s(R)$
exists and is $0$. We prove this important fact next. For graded rings, it is known that
strong and weak F-regularity are equivalent \cite{LS}.}
\end{remark}

\begin{remark}{\rm Watanabe and Yoshida \cite{WY5} systematically studied minimal
possible difference between the Hilbert-Kunz multiplicity of two $\m$-primary ideals. They
go further, and introduced the notion of minimal relative
Hilbert-Kunz multiplicity $mHK(R)$. By their definition, $mHK(R) = \liminf \la_R(R/ann_R Az^q)$,
where $z$ is a generator of the socle of the injective hull $E_R(k).$ They prove that $mHK(R)\leq e_{HK}(I)-e_{HK}(I')$ for
$\m$-primary ideals $\subset I'$ with $\la_R(I'/I) = 1$. 
If $R$ is Gorenstein, they prove the minimal relative \HK multiplicity is in fact $e_{HK}(J) - e_{HK}(J : m)$ 
for any parameter ideal $J$ of $R$. 
As an example, we quote one of their
theorems:
Let $k$ be a field of characteristic $p > 0$, and let $R = k[x_1,\ldots,x_d]^G$ be
the invariant subring by a finite subgroup $G$ of $GL(d, k)$ with $(p, |G|) = 1$. Also, assume
that $G$ contains no pseudo-reflections. Then the minimal relative \HK multiplicity is $1/|G|$.}
\end{remark}

\begin{lemma}\label{positiveF-sig}  Assume that $(R, \m)$ is a 
reduced F-finite local ring
containing a field of prime characteristic $p$ and let $d = \dim
R$. We adopt the notation from the beginning of this section.
If $s(R) > 0$, then $R$ is weakly
$F$-regular.
\end{lemma}

\begin{proof}  Assume that $s(R) > 0$, but  $R$ is not weakly $F$-regular, that is, not
all
ideals of $R$ are tightly closed. By \cite[Theorem 6.1]{HH2}
$R$ has a test element, and then \cite[Proposition 6.1]{HH1}
shows that the tight closure of an arbitary ideal in $R$ is the intersection of
$\m$-primary tightly closed ideals. Since $R$ is not weakly $F$-regular, there
exists an $\m$-primary ideal $I$ with
$I
\neq I^*$. Choose an element $\Delta$ of $I : \m$ which is not in
$I^*$. 

$$q^{\alpha(R)}(\lambda(R/I^{[q]}) - \lambda(R/(I,\Delta)^{[q]})) 
        =\lambda(R^{1/q}/I R^{1/q}) - \lambda(R^{1/q}/(I,\Delta)R^{1/q})\geq a_q$$

Dividing by $q^{d+\alpha(R)}$ and taking the limit gives on the left-hand side
a difference of Hilbert-Kunz multiplicities,
$$e_{HK}(I) - e_{HK}((I,\Delta))\geq s(R).$$
But by Theorem~\ref{HK-tc}, this difference is zero, showing that $s(R) = 0$.
\end{proof}

The beautiful idea of Tucker's proof that the F-signature exists in general is to represent it as a
limit of certain normalized Hilbert-Kunz multiplicities, which are decreasing. To 
capture this, we first discuss some general facts about free summands of modules.

\begin{discuss} {\rm Let $(R,\m)$ be a Noetherian local reduced ring, and let $M$ be a
torsion-free $R$-module. We can always write $M = N\oplus F$, where $F$ is free and
$N$ has no free summands. We define a submodule $M_{nf}$ of $M$ to be $N+\m F$. On
the face of it, this submodule depends on the choice of $N$. However we can also
describe this submodule by the following:

$$ \{x\in M|\, \phi(x) \in \m\, \forall\, \phi \in  \Hom_R(M, R)\}.$$

To see that these are the same, simply note that clearly $M_{nf}$ is inside the
above submodule (note it is a submodule!), and conversely, if $x$ is in the submodule,
then $x\in M_{nf}$; otherwise we can write $x = n + y$, where $y$ is a minimal generator
of $F$, and where $n\in N$. The submodule $Ry$ of $M$ clearly splits off as a free summand,
so there is a $\phi: M\ra R$ such that $\phi(y) = 1$. Then $\phi(x) = 1 + \phi(n)\notin \m$,
a contradiction. Note that $M/M_{nf}$ is a vector space of dimension equal to the rank of $F$.} 
\end{discuss}
                                                                             
\begin{define} Let $(R,\m,k)$ be a reduced local Noetherian ring of prime characteristic $p$.
For $q = p^e$, we let $I_q: = (R^{1/q})_{nf}^{[q]}$, an ideal in $R$.
\end{define}

This ideal was considered in work of  Yongwei Yao \cite{Y1} as well as Florian Enescu and Ian Aberbach
\cite{AE2}. Observe that Tucker defines it as follows, which from the discussion above is
equivalent to our definition:
                                         
$$  I_q = \{r\in R |\, \phi(r^{1/q}) \in \m\, \forall\, \phi \in  \Hom_R(R^{1/q}, R)\}.$$

We group some basic remarks about these ideals in the following proposition:

\begin{prop}\label{ideal-prop}  Let $(R,\m,k)$ be a reduced local Noetherian ring of prime characteristic $p$.
Then $\m^{[q]}\inc I_q$ for all $q = p^e$. Furthermore $I_q^{[q']}\inc I_{qq'}$ for all $q = p^e$ and $q' = p^{e'}$.
If the residue field is perfect, $\la(R/I_q) = a_q$.
\end{prop}

\begin{proof} Since $\m R^{1/q}\inc (R^{1/q})_{nf}$, it is immediate from the definition that $\m^{[q]}\inc I_q$.
To prove the second statement, let $r\in I_q$, so that $r^{1/q}\in (R^{1/q})_{nf}$. Then $(r^{q'})^{1/qq'} = r^{1/q}\in
R^{1/qq'}$ is clearly $I_{qq'}$ by the second description of these ideals, since if $\phi: R^{1/qq'}\ra R$ was such that
$\phi(r^{1/q}) \notin \m$, restricting $\phi$ to $R^{1/q}$ would give the contradiction that $r\notin I_q$. 
The last statement of the proposition follows since $\la(R/I_q) = \la(R^{1/q}/I_q^{1/q}R^{1/q}) = \la(R^{1/q}/(R^{1/q})_{nf}) = a_q$.
\end{proof}

We are ready to prove Tucker's theorem:

\begin{theorem}\label{Fsigexists} \cite[Theorem 4.9]{Tu} Let $(R,\m,k)$ be a Noetherian local ring of dimension $d$ and prime characteristic $p$. Assume
that $R$ is F-finite. Then $s(R) = \varinjlim \frac{a_q}{q^{d+\alpha(R)}}$ exists. 
\end{theorem}

\begin{proof} We can complete $R$ and extend the residue field to assume that $R$ is complete with perfect residue field.
By Lemma \ref{positiveF-sig} if $R$ is not weakly F-regular, then $s(R) = 0$. Hence we may assume that $R$ is weakly F-regular, and is in
particular a Cohen-Macaulay domain.
We use Corollary \ref{uniformity}.  We have that there is a constant $C$ such that for all $q,q'$,
$$ |\la(R/I_q^{[q']}) - (q')^{d}\la(R/I_q)|\leq C(q')^{d}q^{d-1}.$$
Dividing by $(q')^d$ we obtain that
$$ |\la(R/I_q^{[q']})/(q')^d - \la(R/I_q)|\leq Cq^{d-1}.$$
Taking the limit as $q'$ goes to infinity, we see that
$$|e_{HK}(I_q)-a_q|\leq Cq^{d-1}.$$ Dividing by $q^d$ shows that the F-signature exists if and only if the limit
of $e_{HK}(I_q)/q^d$ exists. 
This follows by 
noting that $I_q^{[p]}\inc  I_{qp}$ for all $q$, so that
$e_{HK}(I_{qp})\leq e_{HK}(I_q^{[p]}) = p^de_{HK}(I_q)$, so that
dividing through by $qp$ shows that the sequence $\{e_{HK}(I_q)/q^d\}$ is decreasing,
and thus has a limit, necessarily equal to $s(R)$. \end{proof}

\begin{example}\label{quot-sing-sig} {\rm We return to Example \ref{quot-sing}, where the Hilbert-Kunz
multiplicity of simple quotient singularities were given. Let $(R,\m)$ be a two-dimensional complete
Cohen-Macaulay ring. Assume that $R$ is F-finite and is Gorenstein and
F-rational. Then $R$ is a double point and is isomorphic to $k[[x,y,z]]/(f)$, where $f$
is one of the following:
\medskip

\begin{tabular}{lllll}
\underline{type} & \underline{equation} & \underline{char $R$} & $s(R)$ \\
($A_n$) & $f = xy+z^{n+1}$ & $p\geq 2$ & $1/(n+1)$ & ($n \geq 1$)\\
($D_n$) & $f = x^2+yz^2+y^{n-1}$ & $p \geq 3$ & $1/4(n-2)$ & ($n\geq4$) \\
($E_6$) & $f = x^2+y^3+z^4$ & $p\geq 5$ & $1/24$ \\
($E_7$) & $f = x^2+y^3+yz^3$ & $p \geq 5$ & $1/48$ \\
($E_8$) & $f = x^2+y^3+z^5$ & $p\geq 7$ & $1/120$
\end{tabular}

\medskip
As in Example \ref{quot-sing}, in each of these examples a minimal reduction $J$ of the maximal ideal $\m$
has the property that $\m/J$ is a vector space of dimension $1$. Hence
$e_{HK}(J) - e_{HK}(R) = s(R)$ by Theorem \ref{Fsig-Gor}.
Since $J$ is generated by a regular sequence and is a reduction of
$\m$, $e_{HK}(J) = e(J) = e(\m) = 2$. On the other hand, Example \ref{quot-sing}
gives the Hilbert-Kunz multiplicity for each of these examples, and in each case 
it is $2-1/|G|$, where each ring is the invariant ring of a finite group $G$
acting on a power series ring, giving
our statement. 
Notice that 
the F-signature is exactly $1/|G|$. The same reasoning applies to Example \ref{irrational} to show that
if the \HK multiplicity is irrational in this example, as expected, then so is the F-signature in the same
example.} 
\end{example}

\bigskip

\section{A Second Coefficient}
\medskip

In this section we take up a more careful study of the Hilbert-Kunz function, showing that a
second coefficient exists in great generality. This was proved in \cite{HMM}, and further improved
in \cite{HoY}. The approach we give in this paper is a bit different than those
appearing elsewhere, following an alternate proof developed by Moira McDermott and myself, but not
previously published. The proof in \cite{HMM} relies on the theory of divisors associated to modules.
The approach here rests on growth of Tor modules. In some ways this method is less transparent than
that in \cite{HMM}, but this author believes it has considerable value nonetheless. We are aiming to prove:

\begin{theorem}\label{2coeff}
Let $(R,\m,k)$ be an excellent, local, normal ring of characteristic $p$
with a perfect residue field and $\dim R=d$. Let $I$ be an $\m$-primary ideal.
Then $\la(M/\Iq M)=\alpha q^{d}+\beta q^{d-1} +O(q^{d-2})$ for some $\alpha$ and
$\beta$ in $\mathbb R$.
\end{theorem}

In \cite{HoY}  the condition that $R$ be normal is weakened to just assuming that $R$ satisfies
Serre's condition $R_1$.

One could hope that this theorem  could be generalized to prove that there
exists
a constant $\gamma$ such that
$\la(M/\Iq M)=\alpha q^{d}+\beta q^{d-1} + \gamma q^{d-2} +O(q^{d-3})$ whenever
$R$ is non-singular in codimension two.  However, this cannot be true.
For instance, see Example \ref{HK-biz}.

We first discuss the growth of Tor modules, expanding on what we did in earlier
sections.

\begin{lemma}\label{torsiontor1}
Let $(R,\m,k)$ be a  local ring
of characteristic $p$.
If $T$ is a finitely generated torsion $R$-module with $\dim T = \ell$,
then
$\len{1}{T}\leq O(q^{\ell})$.
\end{lemma}

\begin{proof} Set $d = \dim R$.
Choose a system of parameters $(x_1, \ldots, x_d) \subseteq I$.
We induct on $\lambda(I/(x_1, \ldots, x_d))$.
If $\lambda(I/(x_1, \ldots, x_d))>0$, then there exists $J \subset I$
with $\lambda(I/J)=1$ so that we may write $I=(J,u)$ with $J: u = \m$.
For every $q = p^n$ there is an exact sequence
\[
0 \rightarrow
R/J^{[q]}\col u^q \rightarrow
R/J^{[q]} \rightarrow
R/I^{[q]} \rightarrow
0.
\]
Tensor with $T$ and look at the following portion of the long exact
sequence:
\[
\cdots \rightarrow
\tor_1(R/J^{[q]},T) \rightarrow
\tor_1(R/I^{[q]},T) \rightarrow
\tor_0(R/J^{[q]}\col u^q,T) \rightarrow \cdots.
\]
We have
$\lambda(\tor_1(R/J^{[q]},T))\leq O(q^{d-2})$ by induction.
Also, since $J\col u=\m$, we have ${\m}^{[q]}\subseteq J^{[q]}\col u^q$ and
$\lambda(\tor_0(R/J^{[q]}\col u^q,T)) \leq \lambda(\tor_0(R/{\m}^{[q]},T))$.
But $\lambda(\tor_0(R/{\m}^{[q]},T))$ is the Hilbert-Kunz function for
$T$, so $\lambda(\tor_0(R/{\m}^{[q]},T))=O(q^{\dim T})$ and $\dim T
\leq d-2$.

We have reduced to the case where $\lambda(I/(x_1, \ldots, x_d))=0$.
We need a theorem which is implicitly in Roberts \cite{Ro} and
explicitly given as Theorem 6.2 in  \cite{HH2}:

\begin{theorem}
Let $(R,m)$ be a local ring of
characteristic $p$ and let $G_{\bullet}$ be a finite complex
$$
0\rightarrow G_{n}\rightarrow\dots\rightarrow G_{0}\rightarrow 0
$$
of length $n$ such that each $G_{i}$ is a finitely generated free module
and suppose that each $H_{i}(G_{\bullet})$ has finite length.  Suppose
that $M$ is a finitely generated $R$-module.  Let $ d = \dim  M.$  Then
there is a constant $C > 0$ such that $\ell(H_{n-t}(M\otimes_{R}
F^{e}(G_{\bullet}))\leq Cq^{\m(d,t)}$ for all $t\geq 0$ and all $e\geq 0$, where $q =
p^{e}$. \end{theorem}

Consider $K_{\bullet} ((\underline x);R)$, the Koszul complex on $(x_1,
\ldots, x_d)$.  Let $H_{\bullet}((\underline x);R)$ denote the
homology of the Koszul complex. We apply the above theorem to conclude that there exists a
constant $C>0$ such that
$\lambda(H_{d-t}(T\otimes F^e(K_{\bullet})))\leq Cq^{\min\{\ell,t\}}$ for
all $t$ and for all $e$.
Hence $\lambda(H_i(T\otimes F^e(K_{\bullet})))\leq O(q^{\ell})$ for all $i$.
In general, $H_1(T\otimes F^e(K_{\bullet})))$ maps onto
$\tor_1(T,R/I^{[q]}))$, which gives the stated result.
\end{proof}

Next we study the growth of  $\tor_2$.

\begin{lemma}\label{torsiontor2}
Let $(R,\m,k)$ be a Noetherian local ring of dimension $d$ satisfying Serre's condition
$S_2$, and having prime
characteristic $p$.
Let $T$ be an $R$-module with $\dim T \leq d-2$.
Then $\len{2}{T} = O(q^{d-2})$.
\end{lemma}

\begin{proof}
Pick a regular sequence $x,y$ contained in the annihilator of $T$.
There is an exact sequence
\[
0 \rightarrow
T' \rightarrow
(R/(x,y))^n \rightarrow
T \rightarrow
0
\]
Note $\dim T' = d-2$.
Next tensor with $R/I^{[q]}$ and consider the following portion of the
long exact sequence:
\[
\cdots
\rightarrow
\tors{2}{R/(x,y)}^{\oplus n} \rightarrow
\tors{2}{T} \rightarrow
\tors{1}{T'}\rightarrow \cdots.
\]
Since $x,y$ is regular sequence, we know $\sum_{i=0}^2
\lambda(\tor_i(R/(x,y), R/I^{[q]})) =0$.  Also, \newline
$\len{1}{R/(x,y)} = O(q^{d-2})$ by Lemma \ref{torsiontor1}.
Then $\len{2}{R/(x,y)} = O(q^{d-2})$ as well.
We also know that $\len{1}{T'} = O(q^{d-2})$ by Lemma \ref{torsiontor1}.
 From the long exact sequence above,
 we conclude that $\len{2}{T} = O(q^{d-2})$.
 \end{proof}

 The main surprise is the next lemma, which shows that for the first Tor, modules which
 are torsion-free have slower growth than those which are torsion!

 \begin{lemma}\label{torfreetor1}
 Let $(R,\m,k)$ be a normal  local ring of dimension $d$ and
 prime characteristic $p$.
 Let $M$ be a torsion-free $R$-module.
 Then $\len{1}{M} = O(q^{d-2})$.
 \end{lemma}

 \begin{proof}
 Consider the following exact sequence where $M^* = \Hom_R(M,R)$:
 \[
 0 \rightarrow M \xrightarrow{\theta} M^{**} \rightarrow T \rightarrow 0.
 \]
 Note that $\theta$ is an isomorphism in codimension one and consequently $T$ is a
 torsion-module with $\dim T \leq d-2$.
 We obtain the following long exact sequence:
 \begin{multline*}
 \cdots
 \rightarrow
 \tors{2}{T} \rightarrow
 \tors{1}{M} \rightarrow
 \tors{1}{M^{**}} 
 \rightarrow
 \tors{1}{T} \rightarrow \cdots.
 \end{multline*}
  From this we conclude that
  \begin{align*}
  |\len{1}{M} - \len{1}{M^{**}}| &\leq
  \len{2}{T} + \len{1}{T} \\
  &= O(q^{d-2}).
  \end{align*}
  The last inequality follows from Lemma \ref{torsiontor1} and Lemma
  \ref{torsiontor2}.
  So we may replace $M$ by $M^{**}$ and assume that $M$ has depth 2.
  Therefore, $M$ is $S_2$ and $M_{P}$ is free for all
  height one primes $P$.

We can choose a regular sequence $x,y$ such that they kill all $\tor_i^R(M,\quad)$ for $i\geq 1$.
This can be done in many ways. For example,
we leave as an exercise that there exists a sequence, $x,y$, which is a regular sequence
on $R$ and on
$M$ such that multiplication by $x$ on  $M$ factors through a free
module $F=R^r$ and multiplication by $y$ on $M$
also factors through $F$. These multiplications then induce homotopies which can be used to prove
our claim.

We let $...\ra F_2\ra F_1\ra F_0\ra M\ra 0$ be the start of a minimal free resolution of $M$. 
We tensor with $R/\Iq$, and write $'$ for images after tensoring.  Let $Z_q$ be the kernel
of the induced map from $F'_1$ to $F_0'$, and $B_q$ be the image of the induced map from
$F'_2$ to $F_1'$. Thus, $\tors{1}{M} = Z_q/B_q$. Consider the short exact sequence,
$$0\ra \tors{1}{M}\ra F'_1/B_q\ra N/\Iq N\ra 0,$$
where $N$ is the kernel of the map from $F_0$ onto $M$. 
We tensor with $R/(x,y)$ and use that both $x$ and $y$ annihilate $\tors{1}{M}$ to see that
the length of this Tor is at most $\la(\tor_1(R/(x,y),N/\Iq N)) + \la(F'_1/(B_q+(x,y)F_1'))\leq
\la(\tor_1(R/(x,y),N/\Iq N)) + \la(R/((x,y)+\Iq))\cdot\rank(F_1).$
If we had the term $R/\Iq$ in the first Tor module instead of $N/\Iq N$, we could apply Theorem \ref{torsiontor1} directly to see
the sum is $O(q^{d-2})$. We leave it to the reader to show that this change does not affect the order of growth. \end{proof}

  We record the following two corollaries to Lemma \ref{torfreetor1}.
  \begin{cor}\label{torsiontor2alldim} Let $(R,\m,k)$ be a  local, normal ring of characteristic $p$
  with $\dim R=d$.  Let $M$ be a finitely generated $R$-module.  Then for all $i\geq 2$, $\len{i}{M} = O(q^{d-2})$.
  \end{cor}
  \begin{proof}
   Consider the exact sequence
   $0 \rightarrow {\Omega}^1(M) \rightarrow F \rightarrow M \rightarrow 0$
   where $F$ is free.
   Hence $\len{i}{M} \cong \len{i-1}{{\Omega}^1(M)}$.  It
   follows that to prove the lemma, we need only consider the case $i = 2$, and in this case
   since ${\Omega}^1(M)$ is torsion free, the Lemma above implies that $\len{1}{{\Omega}^1(M)} =
   O(q^{d-2})$, giving that $\len{2}{M} = O(q^{d-2})$.
   \end{proof}

   The next corollary
   shows that $\len{1}{-}$ is additive on short exact sequences of
   torsion modules, up to $O(q^{d-2})$.
   \begin{cor}\label{additive}
   If $T_1$, $T_2$ and $T_3$ are torsion $R$-modules, and $0\rightarrow T_1
   \rightarrow T_2 \rightarrow T_3 \rightarrow 0$ is exact, then
   $|\sum_{i=1}^{3}(-1)^{i+1} \len{1}{T_i}| = O(q^{d-2})$.
   \end{cor}

   \begin{proof}
   After tensoring the exact sequence with $R/I^{[q]}$ we obtain the
   following long exact sequence:
   \begin{multline*}
   \cdots
   \rightarrow
   \tors{2}{T_3} \rightarrow \tors{1}{T_1} \rightarrow \tors{1}{T_2} \rightarrow
   \tors{1}{T_3} \\
   \rightarrow
   \tors{0}{T_1} \rightarrow \tors{0}{T_2} \rightarrow \tors{0}{T_3} \rightarrow
   0
   \end{multline*}
   We examine the cokernel at one spot in the previous sequence.  Consider
   \begin{multline*}
   \rightarrow
   \tors{2}{T_3} \rightarrow \tors{1}{T_1} \rightarrow \tors{1}{T_2} 
   \rightarrow
   \tors{1}{T_3} \rightarrow
   C \rightarrow
   0.
   \end{multline*}
   We know that $\len{2}{T_3} = O(q^{d-2})$ by Corollary
   \ref{torsiontor2alldim}.  It is therefore enough to show that
   $\lambda(C) =
   O(q^{d-2})$.
   We also have the exact sequence
   \[
   0 \rightarrow C \rightarrow \tors{0}{T_1} \rightarrow \tors{0}{T_2} \rightarrow \tors{0}{T_3} \rightarrow
   0.
   \]
Since the $T_i$ are torsion modules, $\dim T_i\leq d-1$, and
there are constants $c_i\geq 0$ such that $\len{0}{T_i} = c_iq^{d-1} + O(q^{d-2})$
so that
\[
\lambda(C) =
c_1q^{d-1}- c_2q^{d-1}+ c_3q^{d-1}+ O(q^{d-2}).
\]
But since the Hilbert-Kunz multiplicity is additive on short exact
sequences, $c_2 = c_1+c_3$, and hence $\lambda(C) = O(q^{d-2})$.
\end{proof}

The next result refines Lemma \ref{torsiontor1} by proving the existence
of a coefficient giving the growth pattern.

\begin{theorem}\label{thmtorsion}
Let $(R,\m,k)$ be an excellent, local, normal ring
of characteristic $p$ with perfect residue field and
with $\dim R=d$.
Let $T$ be a torsion $R$-module.
Then there exists $\gamma(T) \in \mathbb R$ such that
$\len{1}{T} = \gamma(T) q^{d-1} + O(q^{d-2})$.
\end{theorem}

\begin{proof} We may complete $R$ and henceforth assume $R$ is complete. Hence
$R$ is F-finite.

By Corollary \ref{additive}, it is enough to prove the result for $T=R/Q$
where $Q$ is a height one prime of $R$.
If $\dim T \leq d-2$, we know that $\len{1}{T} = O(q^{d-2})$ by Lemma
\ref{torsiontor1} and $\len{2}{T} \leq O(q^{d-2})$ by Lemma
\ref{torsiontor2}.
Let $Q$ be a height one prime of $R$ and consider the following exact
sequence:
\[
0 \rightarrow
(R/Q)^{p^{d-1}} \rightarrow
{}(R/Q)^{1/p} \rightarrow
T \rightarrow
0.
\]
Tensor with $R/I^{[q]}$ and look at the following portion of the
corresponding long exact sequence:
\begin{align*}
\rightarrow
\tors{2}{T} \rightarrow \tors{1}{R/Q}^{p^{d-1}} \rightarrow \tors{1}{{}(R/Q)^{1/p}} \rightarrow
\tors{1}{T} \rightarrow .
\end{align*}
 From this we see that
 \begin{equation}\label{inequality1}
 |p^{d-1}\len{1}{R/Q} - \len{1}{{}(R/Q)^{1/p}}| = O(q^{d-2}),
 \end{equation}

 Next consider the exact sequence
 $0 \rightarrow Q^{1/p} \rightarrow R^{1/p} \rightarrow {}(R/Q)^{1/p}
 \rightarrow 0$.
 First note that $\len{1}{{}R^{1/p}} = O(q^{d-2})$ by Lemma
 \ref{torfreetor1}.
  From the usual long exact sequence on Tor we observe that
  \begin{align*}
  \len{1}{{}(R/Q)^{1/p}} &\leq \len{0}{{}Q^{1/p}} - \len{0}{{}R^{1/p}}\\
  &\qquad + \len{0}{{}(R/Q)^{1/p}} +
  O(q^{d-2}) \\
  &\leq
  \len[pq]{0}{Q} - \len[pq]{0}{R} \\
  &\qquad+ \len[pq]{0}{R/Q} + O(q^{d-2})
  \end{align*}
  Now consider the sequence
  $0 \rightarrow Q \rightarrow R \rightarrow R/Q \rightarrow 0$.
  After tensoring with $R/I^{[pq]}$, it is clear from the usual long exact
  sequence
  that
  \[
  \len[pq]{1}{R/Q}= \len[pq]{0}{Q}- \len[pq]{0}{R}+ \len[pq]{0}{R/Q}.
  \]
  Combining this with the previous inequality shows that
  \begin{equation*}\label{inequality2}
  \len{1}{{}(R/Q)^{1/p}} \leq \len[pq]{1}{R/Q} +O(q^{d-2}).
  \end{equation*}
  Combining \eqref{inequality1} and the previous inequality yields
  \begin{equation*}
  p^{d-1}\len{1}{R/Q} - \len[pq]{1}{R/Q} \leq O(q^{d-2}).
  \end{equation*}

  Recall that $q=p^e$.  Define $\delta_{q}=\len{1}{R/Q}/q^{d-1}$.
  We claim that $\{\delta_q\}$ is a Cauchy sequence.  We use the previous inequality 
  to observe that
  \begin{align*}
  \delta_{pq}- \delta_q &= \len[pq]{1}{R/Q}/{(pq)}^{d-1}- p^{d-1}\len{1}{R/Q}/p^{d-1}q^{d-1} \\
  &= O(1/q)\\
\end{align*}
The sequence $\{\delta_q\}$ converges to some $\gamma(R/Q) \in
\mathbb R$.
A simple argument shows further that
$|\delta_q -\gamma(R/Q)| = O(q^{-1})$.
Hence $\len{1}{R/Q} = \gamma(R/Q)q^{d-1}+O(q^{d-2})$.
\end{proof}

\begin{prop}\label{torfreeprop}
Let $(R,\m,k)$ be an excellent, local, normal ring
of characteristic $p$
with
$\dim R=d$.
Let $M$ be a torsion-free $R$-module of rank $r$.
Then there exists $\gamma(M) \in \mathbb R$ such that
$\len{0}{M} - r \len{0}{R} = \gamma(M) q^{d-1} + O(q^{d-2})$.
\end{prop}

\begin{proof}
We may complete $R$ and henceforth assume $R$ is complete.
Since $M$ is torsion-free of rank $r$ as an $R$-module, we can choose an
embedding $R^r \rightarrow M$ such that the cokernel $T$ is a torsion
module over $R$, and so $\dim T \leq d-1$.
We have the following exact sequence:
$0 \rightarrow R^{r}\rightarrow M \rightarrow T \rightarrow 0.$
Tensor with $R/I^{[q]}$ and consider the usual long exact sequence:
\begin{multline*}
0
\rightarrow
\tors{1}{M} \rightarrow \tors{1}{T} \rightarrow \tors{0}{R}^{\oplus r} \\ \rightarrow
\tors{0}{M} \rightarrow \tors{0}{T} \rightarrow 0.
\end{multline*}
We know that $\len{1}{M} = O(q^{d-2})$ by Lemma \ref{torfreetor1} and
$$\len{1}{T}=\gamma(T)q^{d-1}+O(q^{d-2})$$ by Theorem \ref{thmtorsion}.
Also, $\len{0}{T}$ is the Hilbert-Kunz function for $T$ and
therefore there is a constant $C\geq 0$ such that $\len{0}{T} = Cq^{d-1} +O(q^{d-2})$.  Thus,
\[
\len{0}{M}- r\len{0}{R} = \gamma(M) q^{d-1} + O(q^{d-2})
\]
for some $\gamma(M) \in \mathbb R$.
\end{proof}

\begin{cor}\label{Rcor}
Let $R$ be an excellent, local, normal ring
of characteristic $p$
with perfect residue field
and $\dim R=d$.
Then there exists $\gamma = \gamma({}R^{1/p}) \in \mathbb R$ such that
$\len[pq]{0}{R} - p^d \len{0}{R} = \gamma q^{d-1} + O(q^{d-2})$.
\end{cor}

\begin{proof} We complete $R$ and assume it is complete. Then
$R^{1/p}$ is a finitely generated $R$-module of rank  $p^d$. Thus,
\[
\len{0}{R^{1/p}}-p^d \len{0}{R} = \gamma q^{d-1} + O(q^{d-2})
\]
for some $\gamma \in \mathbb R$ by Proposition \ref{torfreeprop}.
As $\len{0}{R^{1/p}} = \len[pq]{0}{R}$, we have
\[
\len[pq]{0}{R} - p^d \len{0}{R} = \gamma(R^{1/p}) q^{d-1} + O(q^{d-2}).
\]
\end{proof}

The next two theorems are the main content in \cite{HMM}.
As mentioned earlier, the approach in this paper is through divisors attached to modules,
rather than the growth of the length of Tor modules.
See \cite{K2} for further analysis of the second coefficient.
\medskip

\begin{theorem}\label{thmforR}
Let $(R,\m,k)$ be an excellent, local, normal ring of dimension $d$ and prime characteristic $p$ with
a perfect residue field.  Then there exists $\beta(R) \in \mathbb R$ such that
$\lambda(R/I^{[q]}) = e_{HK}(R)q^d + \beta(R) q^{d-1} + O(q^{d-2})$.
\end{theorem}

\begin{proof}
We may complete $R$ and henceforth assume $R$ is complete.
Define
${\epsilon}_q:= \lambda(R/I^{[q]}) - (\gamma(R^{1/p})/(p^{d-1}-p^d)) q^{d-1}$.
Recall that $q=p^e$.
We claim that $\{{\epsilon}_q/q^d\}$ is a Cauchy sequence.
Corollary \ref{Rcor} shows that ${\epsilon}_{pq}-p^d {\epsilon}_q = 
O(q^{d-2})$.
Hence $|{\epsilon}_{pq}/(pq)^d - {\epsilon}_q/q^d| = O(q^{-2})$.
The sequence $\{{\epsilon}_q/q^d\}$ converges to some $\alpha(R) \in 
\mathbb
R$.
Another simple geometric series argument
shows that
$|{\epsilon}_q/q^d - \alpha(R)| = O(q^{-2})$ and so
${\epsilon}_q = \alpha(R)q^d + O(q^{d-2})$.
In other words,
$\lambda(R/I^{[q]})= \alpha(R)q^d + \beta(R) q^{d-1} + O(q^{d-2})$ where
$\beta(R)=\gamma(R^{1/p})/(p^{d-1}-p^d)$. Clearly $\alpha(R) = e_{HK}(R)$ is forced.
\end{proof}

\begin{theorem}\label{mainthm}
Let $(R,\m,k)$ be an excellent, local, normal ring of dimension $d$ and prime characteristic $p$
with a perfect residue field.
Let $M$ be finitely generated $R$-module.
Then there exists $\beta(M) \in \mathbb R$ such that
$\lambda(M/I^{[q]}M) = e_{HK}(M) q^d + \beta(M) q^{d-1} + O(q^{d-2})$.
\end{theorem}

\begin{proof}
We may complete $R$ and henceforth assume $R$ is complete.
Suppose $M$ is a torsion-free $R$-module of rank $r$.
We know that
$\len{0}{M} - r \len{0}{R} = \gamma(M) q^{d-1} + O(q^{d-2})$ for some
$\gamma(M) \in \mathbb R$ by Proposition \ref{torfreeprop}.
By Theorem \ref{thmforR} we know that
$\lambda(R/I^{[q]}) = \alpha(R) q^d + \beta(R) q^{d-1} + O(q^{d-2})$.
Combining these two results gives:
\begin{gather*}
\len{0}{M} - r (\alpha(R) q^d + \beta(R) q^{d-1} + O(q^{d-2})) =
\gamma(M) q^{d-1} + O(q^{d-2}) \\
\len{0}{M} = r\alpha(R) q^d + (r\beta(R)+ \gamma(M)) q^{d-1} + O(q^{d-2}).
\end{gather*}

If $M$ is not torsion-free,
then we have the following exact sequence where $\overline M$
is torsion-free:
\[
0 \rightarrow T \rightarrow M \rightarrow \overline M \rightarrow 0.
\]
Tensor with $R/I^{[q]}$ and consider the usual long
exact sequence
\[
\cdots \rightarrow \tors{1}{\overline M} \rightarrow T/I^{[q]}T \rightarrow M/I^{[q]}M \rightarrow
{\overline M}/I^{[q]}{\overline M} \rightarrow
0.
\]
We know $\len{1}{\overline M} = O(q^{d-2})$ by Lemma
\ref{torfreetor1}.  Also, $\lambda(T/I^{[q]}T)=e_{HK}(T)q^{\dim T}
+O(q^{\dim T -1 })$ and $\dim T \leq d-1$.
Hence the result for $M$ follows from the result for $\overline M$.
\end{proof}
\bigskip

\section{Estimates on Hilbert-Kunz Multiplicity}
\medskip
In this section we discuss estimates of the Hilbert-Kunz multiplicity.
A key motivating idea in this process was introduced in the paper of Blickle and Enescu \cite{BE} which proved that
for rings which are
not regular, the \HK multiplicity is bounded away from $1$ uniformly. This is the content
of Proposition \ref{BE-est}, which gives the lower bound of $1 + \frac{1}{p^dd!}$ for \unmixed non-regular rings.
However, it was felt that the presence of the characteristic $p$
in the formula bounding the \HK multiplicity away from $1$ should not be necessary. 
Watanabe and Yoshida \cite{WY4} made this explicit with the following conjecture:

\begin{conj}\label{WYconj} Let $d \ge 1$ be an integer and $p > 2$ a prime number.
Put $R_{p,d} := \overline{F_p}[[x_0,x_1,\ldots,x_d]]/(x_0^2+\cdots + x_d^2).$  
Let $(R,\m,k)$ be a $d$-dimensional unmixed local ring with $k =\overline{F_p}$,
an algebraic closure of the field with $p$-elements.  Then the following statements hold.
\begin{enumerate} 
\item If $R$ is not regular, then $e_{HK}(R) \geq e_{HK}(R_{p,d}) \geq 1+a_d$,
where $a_d$ is the dth coefficient of the power series expansion of sec$(x)$ $+$ tan$(x)$ around $0$.
\item If $e_{HK}(R) = e_{HK}(R_{p,d})$, then the $\m$-adic completion $\widehat{R}$ of $R$ is isomorphic to
$R_{p,d}$ as local rings. 
\end{enumerate}
\end{conj}

There are several methods which have been used to estimate the \HK multiplicity. Perhaps the most
effective method is due to Watanabe and Yoshida, the method of estimation by computing volumes. Closely
related ideas were also introduced by Hanes \cite{Ha}.
We illustrate this method in the simplest case where $R$ is a Cohen-Macaulay local ring
of dimension $2$. Higher dimensional cases are of course more difficult, but the basic volume estimates
are similar.
The point is to estimate
$l_A(\m^{[q]}/J^{[q]})$
(where $J$ is a minimal reduction of $\m$) using
volumes in $\mathbb R^d$.
In a later paper, Watanabe and Yoshida use the methods, somewhat refined, to study higher dimension. In
\cite{WY4}, they prove their conjecture up to dimension four. Aberbach and Enescu \cite{AE4} have extended
this by verifying the first part of the conjecture up to dimension six. Dimension seven is open as of the time this article
was written.

We need the following lemma to prove Theorem \ref{dim2-est}.
Just as in \cite{WY1}, it is convenient to adopt the following notation: if $t$ is a real number, then
$I^t: = I^{\lfloor t\rfloor}$.

\begin{lemma} \label{Volume}
Let $(R,\m,k)$ be an unmixed local ring of $\dim R = 2$, 
of prime characteristic $p$, and infinite residue field.
Let $J$ be a parameter ideal of $R$. Let $1\leq s < 2$. Then
we have the following limits:
\[
\lim_{q \to \infty} \frac{\la(R/J^{sq})}{q^2} 
= \frac{e(J) s^2}{2},
\qquad
\lim_{q\to \infty} \la
\left(\frac{J^{sq}+(J^*)^{[q]}}{J^{[q]}}\right)
= e(J) \cdot \frac {(2-s)^2}{2}
\]
\end{lemma}

\begin{proof} We leave these for the reader as an exercise. The first follows from the usual
Hilbert-Samuel multiplicity, while the second can be immediately reduced to the case in
which $R$ is a power series ring and the parameters are regular parameters. In this case the
second limit can be thought of as computing a certain volume. We will describe the $d$-dimensional
case after proving the theorem. \end{proof} 

\begin{theorem}\label{dim2-est}\cite[Corollary]{WY1}
Let $(R,\m,k)$ be a two-dimensional Cohen--Macaulay
local ring of prime characteristic $p$.
Put $e=e(R)$, the multiplicity of $R$.
Then the following statements hold$:$
\begin{enumerate}
\item $e_{HK}(R) \ge \frac{e+1}{2}$.
\item Suppose that $k = \overline{k}$.
Then $e_{HK}(R) = \frac{e+1}{2}$ holds
if and only if the associated graded ring
$gr_{\m}(R)$ is isomorphic to the Veronese
subring $k[X,Y]^{(e)}$.
\end{enumerate}
\end{theorem}

\begin{proof} We will only prove the first statement. We claim that
$$e_{HK}(R)\geq \frac{r+2}{2r+2}e,$$ where $e$ is the multiplicity of $R$, and  $r$ is the minimal number of generators of $\m/J^*$.
The theorem follows easily from this inequality, since the fact that $e\geq r-1$ implies that
$\frac{e+1}{2}\leq \frac{r+2}{2r+2}e.$  

To prove the above claim, we let $s$ be a real number, $1\leq s < 2$.
We may assume that the residue field is infinite, and we then choose a
minimal reduction $J$ of the maximal ideal. Note that $\la(\m^{[q]}/(J^*)^{[q]}) = eq^2-e_{HK}(R)q^2 +O(q)$,
by the tight closure characterization of the Hilbert-Kunz multiplicity, Theorem \ref{HK-tc}, and Theorem \ref{HKexists}. 

We have the following:
$$\la(\m^{[q]}/(J^*)^{[q]}) \leq$$
$$ \la((\m^{[q]} + \m^{sq})/((J^*)^{[q]} + \m^{sq})) +
\la(((J^*)^{[q]} + \m^{sq})/((J^*)^{[q]} + J^{sq}) + \la(((J^*)^{[q]} + J^{sq})/\Jq).$$
The middle term in this sum is neglible, since $J$ is a reduction of $\m$, so that there is a fixed power
of $\m$ annihilating these modules, and the number of generators of a power of $\m$ grows as $O(q)$. Hence
the entire term in $O(q)$. 

We prove that $$\la((\m^{[q]} + \m^{sq})/((J^*)^{[q]} + \m^{sq})) = r \cdot 
\la(R/J^{(s-1)q}) + O(q).$$
By our assumption, we can write as
$\m= J^*+ Ru_1+\cdots +Ru_r$.
Since $J^{(s-1)q}u_i^q \subseteq 
\m^{sq} \subseteq \m^{sq}+(J^*)^{[q]}$,
we have
$$\la\left(\frac{\m^{[q]}+\m^{sq}}{(J^*)^{[q]}+\m^{sq}}\right)
\leq \sum_{i=1}^r \la\left(R/((J^*)^{[q]}+\m^{sq}):u_i^q \right) \le  r \cdot \la(R/J^{(s-1)q}).$$

Also, we have         
$\la((J^*)^{[q]}/J^{[q]}) = O(q^{d-1})$ by Theorem \ref{HK-tc}.
Hence                 
\[                    
\la(\m^{[q]}/(J^*)^{[q]}) \le 
r \cdot \la(R/J^{(s-1)q})
+\la\left(\frac{(J^*)^{[q]}+J^{sq}}{J^{[q]}} \right) +O(q).
\]
Dividing by $q^2$ and letting $q$ go to infinity, it follows from Lemma \ref{Volume} that
$$\ehk(J) - \ehk(\m)\le  
r\cdot e \cdot \frac{(s-1)^2}{2} + e \cdot \frac{(2-s)^2}{2}.$$
Setting $s = \frac{r+2}{r+1}$ proves the claim and finishes the proof of the theorem \end{proof}

\medskip

The more general situation is as follows. We take the next discussion directly from \cite{WY4}.
For any positive real number $s$, we put
\[
v_s := \vol \left\{(x_1,\ldots,x_d) \in [0,1]^d 
\,\bigg|\, \sum_{i=1}^d x_i \le s\right\}, \quad
v_s':=1-v_s,
\]
where $\vol(W)$ denotes the volume of 
$W \subseteq \mathbb R^d$.
With this notation, a key theorem in the work of Watanabe and Yoshida is the following:
\begin{theorem} \label{Key}
Let $(R,\m,k)$ be an unmixed local ring of 
characteristic $p>0$.
Put $d = \dim R \ge 1$.
Let $J$ be a minimal reduction of $\m$, and let
$r$ be an integer with $r \ge \mu_R(\m/J^{*})$, 
where $J^{*}$ denotes the tight closure of $J$.
Also, let $s \ge 1$ be a rational number.
Then we have
\begin{equation} \label{KeyEq}
\ehk(R) \ge e(R)
\left\{v_s - r \cdot \frac{(s-1)^d}{d!}\right\}.
\end{equation}
\end{theorem}

This has been extended in \cite{AE4}.

\begin{example}{\rm [{\rm cf. \cite{BC,WY1}}]  \label{Hyp}
Let $(R,\m,k)$ be a hypersurface local ring of
characteristic $p>0$ with $d = \dim R \ge 1$.
Then
\[
\ehk(R) \ge \beta_{d+1} \cdot e(R), 
\]
where $\beta_{d+1}$ is given by the formula$:$
$$\vol \left\{\underline{x} 
\in [0,1]^d \,\bigg|\, \frac{d-1}{2} \le
\sum x_i
\le \frac{d+1}{2}\right\}
=1 - v_{\frac{d-1}{2}} - v_{\frac{d+1}{2}}'.$$

The first few values of $\beta_{d+1}$, beginning at $d = 0$ are the following:
$1,1,\frac{~3~}{4}, \frac{2}{~3~}, \frac{115}{192}$, and for $d=5$, $\frac{11}{20}$.
}
\end{example}

\begin{exercise}{\rm  {\rm (\cite[Theorem (2.15)]{WY1} }
\label{grHK}
Let $(R,\m,k)$ be a local ring of characteristic $p >0$.
Let $G = gr_{\m}(R)$ the associated graded ring of $R$
with respect $\m$ as above.
Then $\ehk(R) \le \ehk(G_{\mathfrak M}) \le e(R)$.
Give an example to show that equality does not necessarily
hold. (In fact, it seldom holds.)}
\end{exercise}

Our final bounds rest on another technique, due to Aberbach and Enescu, as refined by Celikbas, Dao, Huneke, and Zhang,
which allow one to give a uniform lower bound on the Hilbert-Kunz functions of non-regular rings.
The basic idea of Aberbach and Enescu is to adjoint roots of elements in some fixed minimal reduction of the maximal
ideal. In a bounded number of steps of such adjunctions, one reaches a ring which is not F-rational. In this case as we
have seen, there are good lower bounds for the \HK multiplicity. This reduces the problem to understanding the
relationship between \HK multiplicity of a ring and the ring adjoined some root. At this point the estimates in
\cite{CDHZ} are helpful.
The first uniform bound was given in \cite{AE3}:

\begin{theorem}[Aberbach-Enescu]\label{a-ebound} Let $(R,\m,k)$ be an unmixed ring of dimension $d\geq 2$ and
prime characteristic
 $p$. 
 If $R$ is not regular, then
 $$e_{HK}(R)\geq 1 + \frac{1}{d(d!(d-1)+1)^d}.$$
 \end{theorem}

This bound was improved in the paper \cite{CDHZ} as we describe below. The essential new idea is in the following
proposition:
\medskip

\begin{prop} \label{ranklemma} Let $R$ be a local Noetherian domain, and
let $I=(J,u)$ where $J$  is an integrally closed $\m$-primary ideal of $R$ and $u\in Soc(J)$.
If $M$ is a finitely generated torsion-free $R$-module, then $$\ell(IM/JM)\geq \rank M. $$
\end{prop}

\begin{proof}
Set $N=(JM:_Mu)$. Since $\displaystyle{\frac{M}{N}\cong \frac{(J,u)M}{JM}}$ and $\mathfrak{m}M\subseteq N$, we can write $M=N+N'$ with $\displaystyle{\mu(N')=\ell\left(\frac{IM}{JM}\right)}$.
Thus it suffices to prove $\mu(N') \geq \rank(M)$. Since $u(M/N')\subseteq J(M/N')$, it follows from the
 determinantal trick \cite[2.1.8]{SH} that there is an element $ r = u^n+j_{1} \cdot u^{n-1}+\cdots + j_{n}$
with $j_{i}\in J^{i}$ for all $i$ such that
$rM \subseteq N'$.
 Observe that $r \neq 0$ since $J$ is integrally closed and $u\notin J$. Since
$M_r = N'_r$, this implies that   $\mu(N')\geq \rank(N')=\rank(M)$.
\end{proof}

Given two ideals $I$ and $J$ with $J \subseteq I$, $\overline{\ell}(I/J)$ will denote the longest chain of integrally closed ideals between $J$ and $I$.

\begin{cor} \label{rankcor} Let $R$ be a Noetherian local domain.  Let $J$ be an integrally closed $\mathfrak{m}$-primary ideal of $R$ and let $I$ be an ideal containing $J$.
If $M$ is a finitely generated torsion-free $R$-module, then $$\ell(IM/JM)\geq \overline{\ell}(I/J) \cdot \rank(M).$$
\end{cor}

\begin{proof} Set $n = \overline{\ell}(I/J)$. Then  there is a chain of ideals $$J=K_0\subset K_1\subset \ldots \subset K_{n-1}\subset K_{n}=I$$ with $\overline{K_{i}}=K_{i}$ for all $i$. Then
$$\ell\left(IM/JM\right)\geq \sum_{j=0}^{n} \ell(K_{j+1}M/K_{j}M) \geq \sum_{j=0}^{n} \ell((K_{j},u_{j})M/K_{j}M)$$
for some $u_{j}\in K_{j+1} \cap Soc(K_{j})$. Thus the result follows from Proposition \ref{ranklemma}.
\end{proof}

One of the important ideas in proving that \HK multiplicity equal to one implies regularity was showing an inequality $e_{HK}(I)\geq \la(R/I)$ for
a suitable $\m$-primary ideal $I$. Recall that must have equality if $R$ is regular. This idea was developed 
in \cite[2.17]{WY1}, where the following questions were raised:

\smallskip

{\it  Let $R$ be a Cohen-Macaulay local ring of characteristic
$p > 0$. Then for any $\m$-primary ideal $I$, do we have
    (1) $e_{HK}(I) \geq \ell(R/I)?$
    (2) If pd$_R(R/I) < \infty$, is $e_{HK}(I) = \ell(R/I)$? }

\smallskip

The answer to both questions turns out to be negative; for example, see the paper of Kurano \cite{K1}.  The next exercise shows that
(1) is true for many $\m$-primary ideals \cite{CDHZ}:

\begin{exercise}\label{Wat}{\rm Assume $R$ is an excellent normal ring with an algebraically closed residue field. If $I$ is an integrally closed $\mathfrak{m}$-primary ideal of $R$, then
 $$e_{HK}(I) \geq \ell(R/I)+e_{HK}(R)-1.$$
If $I$ is an $\m$-primary ideal such that there is an integrally closed ideal $K\subset I$ with
$\ell(I/K) = 1$, then
$$e_{HK}(I) \geq \ell(R/I).$$
(Hint: Use \cite[2.1]{Wa} and Corollary \ref{rankcor}.)}
\end{exercise}

We turn to better uniform lower bounds for the \HK multiplicity. An important point is the following, which we leave as an exercise
(see \cite{CDHZ}):

\begin{exercise} \label{intprop} Assume $R$ is Cohen-Macaulay and normal, and let $x\in \mathfrak{m}-\mathfrak{m}^2$ be part of a minimal reduction
 of $\mathfrak{m}$. Let $S=R[y]$ with $y^n=x$. Then $\m S+(y^i)$ is integrally closed for any nonnegative integer $i$.
\end{exercise}

\begin{cor}\label{roots}
Assume that $(R,\m,k)$ is a Cohen-Macaulay normal local ring of prime characteristic $p$ with
infinite residue field. Let $x\in \mathfrak{m}-\mathfrak{m}^2$
 be part of a minimal reduction of $\mathfrak{m}$ and let $S=R[y]$ with $y^n=x.$ Then $$e_{HK}(R)-1\geq \frac{e_{HK}(S)-1}{n}.$$
\end{cor}

\begin{proof}
It follows from Proposition \ref{intprop} and Corollary \ref{rankcor} that $$e_{HK}(\m S)\geq \ell(S/\m S)+e_{HK}(S)-1$$
Note that $S/\mathfrak{m}S \cong k[y]/(y^{n})$. So $\ell(S/\mathfrak{m}S)=n$. Moreover,
$e_{HK}(\mathfrak{m}S)=
n\cdot e_{HK}(R)$ by Theorem \ref{thmmultfiniteext}.
Therefore, $$n\cdot e_{HK}(R)\geq n+e_{HK}(S)-1$$ and hence the result follows.
\end{proof}

We can now give a rough lower bound on the \HK multiplicity of non-regular local ring, which depends only upon the
dimension of the ring.

\begin{theorem} Let $(R,\m,k)$ be a formally unmixed Noetherian local ring of prime characteristic $p$, multiplicity $e > 1$,
and dimension $d$. Then $e_{HK}(R)\geq 1 + \frac{1}{d!d^d}$. \end{theorem}

\begin{proof} If $e_{HK}(R)\geq 1+ 1/d!,$ there is nothing to prove. Hence we may assume that
$e_{HK}(R)< 1+ 1/d!,$ and  then $R$ is $F$-regular and Gorenstein by \cite[3.6]{AE3} (see Proposition \ref{nonFrat-est}
as well). Thus we may assume that $R$ is $F$-rational and Gorenstein.

Let $(\underline{x})=(x_1,\cdots,x_d)$ be a minimal reduction of $\m$.
Consider the set of overrings $S=R[x^{1/n}_1, \dots, x_{i}^{1/n}]=R_{i,n}$ which are
 not $F$-rational. Choose $n$ and $i$ such that we attain
 $\text{min } \{n^i: R_{i,n} \text{ is not F-rational} \}$. Set $S = R_{i,n}$.
 Then by Proposition \ref{nonFrat-est} applied to $x_1^{1/n},...,x_i^{1/n}, x_{i+1},...,x_d$,
 $$e_{HK}(S)\geq \frac{e(S)}{e(S)-1}.$$

 However, since $S/(x_1^{1/n},...,x_i^{1/n}, x_{i+1},..,x_d)\cong R/(\underline{x}),$ we have $e(S)=e$. 
Therefore,
 $e_{HK}(S)\geq 1 + \frac{1}{e-1}.$

 Let $R_0=R,$ and for each $i\geq j\geq 1,$ let $R_j=R_{j-1}[x^{1/n}_j],$ then by Corollary \ref{roots}, $$e_{HK}(R_j)-1\geq \frac{e_{HK}(R_{j-1})-1}{n}.
 $$
 Since $e-1 < d!$,  it remains to prove that
 $$\text{min } \{n^i: R_{i,n} \text{ is not F-regular} \}\leq d^d.$$

 To do this we note that it suffices to prove that $R[x_1^{1/d},...,x_d^{1/d}]$
 is not F-regular. Set $y_i=x_{i}^{1/d}$.
 Then a socle representative  of $S/(\underline{x})$ is 
  $u \cdot y_{1}^{d-1} \ldots y_{d}^{d-1},$
  where $u$ generates the socle of $(\underline{x}R)$. 
  Let $v$ be any discrete valuation centered on the maximal ideal of $S$. Then we claim that
  $$v(u\cdot y_{1}^{d-1} \ldots y_{d}^{d-1})\geq dv(\m).$$
  Since $v(u)\geq v(\m)$, this is clear.

  It follows that $u\cdot y_{1}^{d-1} \ldots y_{d}^{d-1}\in
   \overline{(\m S)^d}$.
   By the tight closure Brian\c con-Skoda theorem \cite[Section 5]{HH1} this implies that 
   $(x_1,...,x_d)S$
   is not tightly closed, which gives the desired conclusion.
   \end{proof}

Another approach, closely related to the volume methods of Watanabe and Yoshida, was given by Douglas Hanes in \cite{Ha}. We close
this survey with some of his results. See in particular \cite[Theorem 2.4]{Ha} and \cite[Corollary 2.8]{Ha}.

\begin{theorem}  Let $(R, \m, k)$ be a Noetherian local ring  of prime characteristic $p$, and dimension $d\geq 2$. 
Let $I$ be an $\m$-primary ideal, and set $t = \mu(I)$. Then,
$$e_{HK}(I)\geq \frac{e(I)}{d!}\cdot \frac{t}{(t^{1/(d-1)}-1)^{d-1}}.$$
\end{theorem}

\begin{proof} We note that $\Iq \inc I^q$  for all $q = p^e$, and $\mu(\Iq)\leq t$ for all $q$. Hence, for all $q = p^e$ and any
$s \in \mathbb N$, $\la((\Iq + I^{q+s})/I^{q+s})\leq t\cdot \la(R/I^s)$.
Therefore, for all $q = p^e$ and any $s \in \mathbb N$, we see that
$$\la(R/\Iq)\geq \la(R/(\Iq + I^{q+s}))\geq \la(R/I^{q+s})-t\cdot \la(R/I^s).$$

Just as in the work of Watanabe and Yoshida, the key point is to choose $s$ carefully. Set $s = q\alpha$.
We obtain that
$$(\frac{e(I)}{d!})[(q+q\alpha)^d - t(q\alpha)^d]\leq \la(R/\Iq) + O(q^{d-1}).$$
Ignoring  the $O(q^{d-1})$ term and computing the maximal value of the function on the left-hand side of this equation,
we obtain that a maximum is achieved when  $\alpha = \frac{1}{(t^{1/(d-1)}-1)^{d-1}}$. The best lower bound for  $e_{HK}(I)$
is obtained by setting $s = \lfloor\frac{q}{(t^{1/(d-1)}-1)^{d-1}}\rfloor$.
Note that $s > 0$, since $t\geq d\geq 2$.
We may write $s = q(\alpha-\epsilon)$
where $\epsilon < 1/q$.
    Applying the equations above with this value of $s$ gives us that
    $$\la(R/\Iq)\geq (\frac{e(I)}{d!})q^d[(1+\alpha-\epsilon)^d - t(\alpha-\epsilon)^d] + O(q^{d-1}).$$
Dividing through by $q^d$, and letting $q$ go to infinity (and $\epsilon$ toward 0),
we obtain the estimate
$$e_{HK}(I)\geq (\frac{e(I)}{d!})[(1+\alpha)^d - t(\alpha)^d],$$
from which the theorem follows. \end{proof}

\begin{cor} 
Let $(R, \m, k)$ be a $d$-dimensional hypersurface ring of prime characteristic $p$,
where $d\geq  3$. Then $e_{HK}(R)\geq  e(R)2^{d-1}/d!$.
\end{cor}

\begin{proof} Apply the previous theorem. Notice that the function $F(t) = \frac{t}{(t^{1/(d-1)}-1)^{d-1}}$ is
decreasing, and $F(2^{d-1}) = 2^{d-1}$. As long as $\mu(\m)\leq 2^{d-1}$ we can then apply the
theorem. Since $\mu(\m)\leq d+1$ and $d\geq 3$, the inequality holds. \end{proof}


\end{document}